\documentclass[a4paper,12pt]{amsart}
\usepackage[hmarginratio={1:1},vmarginratio={1:1},heightrounded,textwidth=400pt]{geometry}
\usepackage{amsmath,amssymb,amsthm}

\usepackage{amsfonts}
\usepackage{hyperref}
\usepackage{tikz}
\usetikzlibrary{calc}
\usetikzlibrary{decorations.markings}
\usetikzlibrary{arrows}
\usetikzlibrary{patterns}
\usetikzlibrary{shapes,snakes}
\usepackage[textsize=small]{todonotes}
\usepackage{url}

\usepackage{todonotes}

\newtheorem{theorem}{Theorem}

\newtheorem{lemma}[theorem]{Lemma}

\let\leq\leqslant
\let\geq\geqslant

\let\grminus\setminus

\let\emptyset\varnothing

\newcommand{\brac}[1]{{\left(#1\right)}}
\newcommand{\sbrac}[1]{{\left[#1\right]}}
\newcommand{\set}[1]{\left\{#1\right\}}
\newcommand{\norm}[1]{{\left|#1\right|}}

\newcommand{\Nat}{\mathbb{N}}

\newcommand{\skel}{\operatorname{skel}}
\newcommand{\ext}{\operatorname{ext}}
\newcommand{\inter}[1]{\operatorname{int}\sbrac{#1}}
\newcommand{\minn}{\operatorname{minn}}
\newcommand{\maxn}{\operatorname{maxn}}

\makeatletter
\newcommand\ie{i.e\@ifnextchar.{}{.\@}}
\newcommand\etc{etc\@ifnextchar.{}{.\@}}
\newcommand\etal{et~al\@ifnextchar.{}{.\@}}
\makeatother

\newcounter{Cases}
\newcounter{CasesSub}
\newcounter{CasesSubSub}
\newcounter{CasesAll}
\makeatletter
\newcommand\caselab[1]{\def\@currentlabel{#1}}
\makeatother
\newcommand{\caselabel}[1]{\refstepcounter{CasesAll}\caselab{#1}}
\newcommand{\Cases}{\setcounter{Cases}{0}\setcounter{CasesSub}{0}\setcounter{CasesSubSub}{0}}
\newcommand{\Case}[1]{\stepcounter{Cases}\setcounter{CasesSub}{0}\setcounter{CasesSubSub}{0}\medskip\noindent\textbf{Case~\arabic{Cases}.\,\,#1.\\\noindent}\caselabel{\arabic{Cases}}}
\newcommand{\Subcase}[1]{\stepcounter{CasesSub}\setcounter{CasesSubSub}{0}\smallskip\noindent\textbf{Case~\arabic{Cases}.\arabic{CasesSub}.\,\,#1.\\\noindent}\caselabel{\arabic{Cases}.\arabic{CasesSub}}}
\newcommand{\Subsubcase}[1]{\stepcounter{CasesSubSub}\noindent\textbf{Case~\arabic{Cases}.\arabic{CasesSub}.\arabic{CasesSubSub}.\,\,#1.\\\noindent}\caselabel{\arabic{Cases}.\arabic{CasesSub}.\arabic{CasesSubSub}}}

\title{Defective 3-Paintability of Planar Graphs}

\author[G.~Gutowski]{Grzegorz Gutowski}
\address[G.~Gutowski, T.~Krawczyk]{Theoretical Computer Science Department, Faculty of Mathematics and Computer Science, Jagiellonian University, Krak\'ow, Poland}
\email{\{gutowski,krawczyk\}@tcs.uj.edu.pl}
\author[M.~Han]{Ming Han}
\address[M.~Han]{School of Mathematical and Statistical Sciences, Arizona State University, Tempe, AZ 85287, USA}
\email{mhan31@asu.edu}
\author[T.~Krawczyk]{Tomasz Krawczyk}
\author[X.~Zhu]{Xuding Zhu}
\address[X.~Zhu]{Department of Mathematics, Zhejiang Normal University, China, and Department of Applied Mathematics, National Sun Yat-sen University, Taiwan}
\email{xdzhu@zjnu.edu.cn}

\thanks{
G. Gutowski is partially supported by Polish National Science Center UMO-2011/03/D/ST6/01370. 
T. Krawczyk is partially supported by Polish National Science Center UMO-2015/17/B/ST6/01873.
}

\begin{document}

\begin{abstract}
A $d$-defective $k$-painting game on a graph $G$ is played by two players: Lister and Painter.
Initially, each vertex is uncolored and has $k$ tokens.
In each round, Lister marks a chosen set $M$ of uncolored vertices and removes one token from each marked vertex.
In response, Painter colors vertices in a subset $X$ of $M$ which induce a subgraph $G\sbrac{X}$ of maximum degree at most $d$.
Lister wins the game if at the end of some round there is an uncolored vertex that has no more tokens left.
Otherwise, all vertices eventually get colored and Painter wins the game.
We say that $G$ is $d$-defective $k$-paintable if Painter has a winning strategy in this game.
In this paper we show that every planar graph is 3-defective 3-paintable and give a construction of a planar graph that is not 2-defective 3-paintable.
\end{abstract}

\maketitle

\section{Introduction}

All graphs considered in this paper are finite, undirected and contain no loops nor multiple edges.
For every $k \geq 1$, the set $\set{1,\ldots,k}$ is denoted $\sbrac{k}$.
The size of a graph $G$, denoted $\norm{G}$, is the number of vertices in $G$.
For a vertex $v$ of $G$, the set of vertices adjacent to $v$ in $G$ is denoted $N(v)$.
For a set $X$ of vertices of $G$, the graph induced by $X$ in $G$ is denoted $G[X]$.

A \emph{$d$-defective coloring} of a graph $G$ is a coloring of the vertices of $G$ such that each color class induces a subgraph of maximum degree at most $d$.
Thus, a 0-defective coloring of $G$ is simply a proper coloring of $G$.
The famous Four Color Theorem asserts that every planar graph is 0-defective 4-colorable.
Defective coloring of graphs was first studied by Cowen, Cowen and Woodall~\cite{CowenCW86}.
They proved that every outerplanar graph is 2-defective 2-colorable and that every planar graph is 2-defective 3-colorable.
They also showed an outerplanar graph that is not 1-defective 2-colorable,
a planar graph that is not 1-defective 3-colorable, and for every $d$, a planar graph that is not $d$-defective $2$-colorable.

A \emph{$k$-list assignment} of a graph $G$ is a mapping $L$ which assigns to each vertex $v$ of $G$ a set $L(v)$ of $k$ permissible colors.
A \emph{$d$-defective $L$-coloring} of $G$ is a $d$-defective coloring $c$ of $G$ with $c(v) \in L(v)$ for every vertex $v$ of $G$.
A graph $G$ is \emph{$d$-defective $k$-choosable} if for any $k$-list assignment $L$ of $G$, there exists a $d$-defective $L$-coloring of G.
The particular function that assigns the set $\sbrac{k}$ to each vertex of a graph is a $k$-list assignment.
Therefore, every $d$-defective $k$-choosable graph is $d$-defective $k$-colorable.
The converse is not true.
Voigt~\cite{Voigt93} gave a construction of a graph that is not $0$-defective $4$-choosable.
Eaton and Hull~\cite{EatonH99} and \v{S}krekovski~\cite{Skrekovski99} independently proved that every planar graph is 2-defective 3-choosable and every outerplanar graph is 2-defective 2-choosable.
They asked the question whether every planar graph is 1-defective 4-choosable.
One decade later, Cushing and Kierstead~\cite{CushingK10} answered this question in the affirmative.

This paper studies the on-line version of list coloring of graphs, defined through a two person game.
The study of on-line list coloring was initiated independently by Schauz~\cite{Schauz09} an Zhu~\cite{Zhu09}.

A \emph{$d$-defective $k$-painting game} on a graph $G$ is played by two players: Lister and Painter.
Initially, each vertex is uncolored and has $k$ tokens.
In each round, Lister marks a chosen set $M$ of uncolored vertices and removes one token from each marked vertex.
In response, Painter colors vertices in a subset $X$ of $M$ which induce a subgraph $G\sbrac{X}$ of maximum degree at most $d$.
Lister wins if at the end of some round there is an uncolored vertex with no more tokens left.
Otherwise, after some round, all vertices are colored and Painter wins the game.
We say that $G$ is \emph{$d$-defective $k$-paintable} if Painter has a winning strategy in this game.
For a vertex $v$ of $G$, let $\theta(v)$ denote the set of neighbors of $v$ that are colored in the same round as $v$.
Thus, in the $d$-defective painting game we have that for any vertex $v$, $\norm{\theta(v)} \leq d$.
We say that vertices in $\theta(v)$ give defect to $v$.

Let $L$ be a $k$-list assignment of $G$ with colors in the set $\sbrac{n}$.
Consider the following strategy for Lister.
In the $i$-th round, for $i \in \sbrac{n}$, Lister marks the set $M_i=\set{v: i \in L(v), v \notin X_1,\ldots,X_{i-1}}$, where $X_j$ is the set of vertices colored by Painter in the $j$-th round.
If Painter wins the game then the constructed coloring is a $d$-defective $L$-coloring of G.
Therefore, every $d$-defective $k$-paintable graph is $d$-defective $k$-choosable.
The converse is not true.
Zhu~\cite{Zhu09} showed a graph that is 0-defective 2-choosable and is not 0-defective 2-paintable.

Thomassen~\cite{Thomassen94} proved that every planar graph is 0-defective 5-choosable and Schauz~\cite{Schauz09} observed that every planar graph is also 0-defective 5-paintable.
As mentioned above, it is known that every planar graph is 2-defective 3-choosable~\cite{EatonH99,Skrekovski99} and 1-defective 4-choosable~\cite{CushingK10}.
Recently, Han and Zhu~\cite{HanZ16} proved that every planar graph is 2-defective 4-paintable.
It remained open questions whether or not every planar graph is 2-defective 3-paintable, or 1-defective 4-paintable.

In this paper, we construct a planar graph that is not 2-defective 3-paintable and prove that every planar graph is 3-defective 3-paintable.
The only remaining question is whether or not every planar graph is 1-defective 4-paintable.

In Section~\ref{sec:painter} we present a strategy for Painter that shows the following.
\begin{theorem}\label{thm:positive}
  Every planar graph is 3-defective 3-paintable.
\end{theorem}
In Section~\ref{sec:lister} we show that this result is best possible as we construct a graph and a strategy for Lister that shows the following.
\begin{theorem}\label{thm:negative}
  Some planar graphs are not 2-defective 3-paintable.
\end{theorem}

\section{Painter's strategy}\label{sec:painter}

In this section we prove Theorem~\ref{thm:positive}.
The proof provides an explicit, recursive strategy for Painter in a 3-defective 3-painting game on any planar graph.
Our proof can be easily transformed into a polynomial-time algorithm that plays the game against Lister.

Let $G$ be a connected non-empty plane graph.
By a plane graph we mean a graph with a fixed planar drawing.
%Vertices of the graph are represented by pairwise different points in the plane.
%Each edge $\set{x,y}$ of the graph is represented by a closed segment connecting points representing vertices $x$ and $y$.
%Any two segments do not intersect each other in points other than the end-points.
Let $C$ be the boundary walk of the outer face of $G$.
For a vertex $v$ in $C$, we define the set of \emph{$C$-neighbors} of $v$ to be the set of vertices that are consecutive neighbours of $v$ in $C$. 
Observe that there may be more than two $C$-neighbors for a single vertex as $C$ is not necessarily a simple walk.
For the purpose of induction, we consider a more general game.
We augment the $3$-defective $3$-painting game and introduce a \emph{$(G,A,b)$-refined game} in which:
\begin{itemize}
  \item $A \cup \set{b}$ are \emph{special} vertices --
    $A$ is a set, possibly empty, of at most two vertices that appear consecutively in $C$;
    $b$ is a vertex in $C$ other than the vertices in $A$.
    There are additional conditions on marking and coloring of special vertices.
  \item each token has a \emph{value} -- when Lister removes a token of value $p$ from a marked vertex $v$ and Painter colors $v$ then at most $p$ neighbors of $v$ are colored in the same round.
    The initial number of tokens of different values will differ from one vertex to another.
\end{itemize}

We say that a vertex $v$ is an \emph{$(A,b)$-cut} if $v \notin A, v \neq b$ and there is a vertex $a$ in $A$ such that $v$ is on every path between $a$ and $b$ in $G$.
We call a vertex in $C$ that is neither in $A$, nor $b$, nor an $(A,b)$-cut to be a \emph{regular boundary} vertex.

Let \emph{token function} $f: V(G) \times \set{0,\ldots,3} \to \Nat$ be a mapping defined for each vertex $v$ and each value between $0$ and $3$.
Initially, each vertex $v$ has $f(v, p)$ tokens of value $p$.
We denote the vector $\brac{f(v,0), \ldots, f(v,3)}$ as $f(v)$.
We set values of $f$ so that:
\begin{itemize}
  \item $f(v) = (0,1,0,0)$ if $v \in A$, or $v = b$,
  \item $f(v) = (0,0,1,0)$ if $v$ is an $(A,b)$-cut,
  \item $f(v) = (0,0,1,1)$ if $v$ is a regular boundary vertex,
  \item $f(v) = (0,0,0,3)$ if $v \notin C$.
\end{itemize}
See Figure~\ref{fig:refined} for an example of a graph and a token function.

In each round, Lister marks a chosen set $M$ of uncolored vertices and removes one token from each marked vertex.
If $\norm{A} = 2$, then Lister is not allowed to mark simultaneously both vertices in $A$, \ie{} $\norm{M \cap A} \leq 1$.
Let $p_v$ denote the value of the token removed by Lister from a vertex $v$ in $M$.
In response, Painter colors vertices in a subset $X$ of $M$ such that the degree of any vertex $v$ in the induced subgraph $G\sbrac{X}$ is at most $p_v$, \ie{} $\forall v \in X:\norm{\theta(v)} \leq p_v$.
Additionally, if $a \in A$, and $\set{a,b}$ is an edge of $C$, then no neighbor of $a$ other than $b$ is colored in the same round as $a$, \ie{} $\theta(a) \subseteq \set{b}$.
Lister wins if at the end of some round there is an uncolored vertex with no more tokens left.
Otherwise, after some round, all vertices are colored and Painter wins.

\begin{figure}[h]

\begin{tikzpicture}[scale=0.9, yscale=0.85]

\tikzset{->-/.style={decoration={
  markings,
  mark=at position #1 with {\arrow{latex}}},postaction={decorate}}}

\coordinate (b) at (-5,0) {};
\coordinate (bc1) at (-2.5,1.5) {};
\coordinate (c1b) at (-2.5,-1.5) {};
\coordinate (a1) at (4.5,0.5) {};
\coordinate (ba1) at (2.5,1.5) {};
\coordinate (a2) at (4.5,-0.5) {};
\coordinate (c1a1) at (2.5,1.5) {};
\coordinate (c2) at (2.5,-1.5) {};
\coordinate (n1) at (4.5,-2.5) {};
\coordinate (n2) at (4.5,-3.5) {};
\coordinate (low) at (2.5,-4.5) {};
\coordinate (left) at (0,-3) {};

\coordinate (v) at (3.5,-3) {};
\coordinate (v1) at (3.5,-3.75) {};
\coordinate (v2) at (2.5,-3.5) {};
\coordinate (v3) at (2.5,-2.5) {};
\coordinate (v4) at (3.5,-2.25) {};

\coordinate (c1) at (0,0) {};

\filldraw[fill=gray!15, draw=none]
(b) to[out=75,in=182] (bc1) to[out=-2,in=105] (c1) to[out=-105,in=2] (c1b) to[out=178,in=-75] cycle;

\filldraw[fill=gray!15, draw=none]
(c1) to[out=75,in=182] (c1a1) to[out=-2,in=135] (a1) to (a2) to[out=-135,in=2] (c2) to[out=178,in=-75] cycle;

\filldraw[fill=gray!15, draw=none]
(left) to[out=88,in=182] (c2) to[out=-2,in=135] (n1) -- (n2) to[out=-135,in=2] (low) to[out=178,in=-88] (left);

\draw[thick] (b) to[out=75,in=182] (bc1) to[out=-2,in=105] (c1);
\draw[thick] (c1) to[out=-105,in=2] (c1b) to[out=178,in=-75] (b);
\draw[thick] (c1) to[out=75,in=182] (c1a1) to[out=-2,in=135] (a1);
\draw[thick] (a1) -- (a2);
\draw[thick] (a2) to[out=-135,in=2] (c2) to[out=178,in=-75] (c1);
\draw[thick] (left) to[out=88,in=182] (c2) to[out=-2,in=135] (n1);
\draw[thick] (n1) -- (n2);
\draw[thick] (n2) to[out=-135,in=2] (low) to[out=178,in=-88] (left);

\draw (n2) -- (v1);
\draw (v1) -- (v2);
\draw (v2) -- (v3);
\draw (v3) -- (v4);
\draw (v4) -- (n1);

\draw (v) -- (n1);
\draw (v) -- (n2);
\draw (v) -- (v1);
\draw (v) -- (v2);
\draw (v) -- (v3);
\draw (v) -- (v4);

\tikzstyle{every node}=[circle,minimum size=5pt,inner sep=0pt,draw,fill]
\begin{tiny}
\node (bn) at (b) {};
\node (a1n) at (a1) {};
\node (a1n) at (a1) {};
%\node (bc1n) at (bc1) {};
%\node (c1bn) at (c1b) {};
\node (a2n) at (a2) {};
\node (c1a1n) at (c1a1) {};
\node (c2n) at (c2) {};
\node (c1n) at (c1) {};

\node (n1n) at (n1) {};
\node (vn) at (v) {};
\node (v1n) at (v1) {};
\node (v2n) at (v2) {};
\node (v3n) at (v3) {};
\node (v4n) at (v4) {};

\node (n2n) at (n2) {};
%\node (lown) at (low) {};
%\node (leftn) at (left) {};

\end{tiny}

\begin{tiny}
\tikzstyle{every node}=[inner sep=0pt]
\node at ($(c1) + (0.3,-0.05)$) {$c_1$};
\node at ($(c2) + (0,0.25)$) {$c_2$};
\node at ($(v) + (-0.15,0.23)$) {$v$};
\end{tiny}

%Game
\tikzstyle{every node}=[inner sep=1pt]
\node (GA1) at ($(a1) + (-0.8,-0.0)$) {$A$};
\node (GA2) at ($(a2) + (-0.8,0.0)$) {$A$};
\node (Gb) at ($(b) + (0.8, 0)$) {$b$};
\draw[->-=1] (GA1) -- (a1n);
\draw[->-=1] (GA2) -- (a2n);
\draw[->-=1] (Gb) -- (bn);

\begin{tiny}
\node (c1f) at ($(c1) + (0,1.3)$) {$(0,0,1,0)$};
\draw[->-=1] (c1f) -- ($(c1n)+(0,0.3)$);

\node (c1a1f) at ($(c1a1) + (0,0.8)$) {$(0,0,1,1)$};
\draw[->-=1] (c1a1f) -- ($(c1a1n)+(0,0.2)$);

\node (c2f) at ($(c2) + (2,0)$) {$(0,0,1,1)$};
\draw[->-=1] (c2f) -- ($(c2n)+(0.7,0)$);

\node (vf) at ($(v) + (2,0)$) {$(0,0,0,3)$};
\draw[->-=1] (vf) -- ($(vn)+(0.3,0)$);

\node (bf) at ($(b) + (-1.4,0)$) {$(0,1,0,0)$};
\draw[->-=1] (bf) -- ($(bn)+(-0.2,0)$);

\node (a1f) at ($(a1) + (1.4,0)$) {$(0,1,0,0)$};
\draw[->-=1] (a1f) -- ($(a1n)+(0.2,0.0)$);

\node (a2f) at ($(a2) + (1.4,0)$) {$(0,1,0,0)$};
\draw[->-=1] (a2f) -- ($(a2n)+(0.2,0.0)$);
\end{tiny}

\end{tikzpicture}
\caption{
An example of $(G,A,b)$-refined game.
Vertex $c_1$ is an $(A,b)$-cut.
Vertex $c_2$ is a regular boundary vertex ($c_2$ is a cut in $G$, but not an $(A,b)$-cut).
Since for every $a \in A$, $\set{a,b}$ is not an edge of $C$, each vertex in $A$ can get one defect from any of its neighbours outside $A$ (vertices of $A$ are not marked simultaneously).
}
\label{fig:refined}
\end{figure}
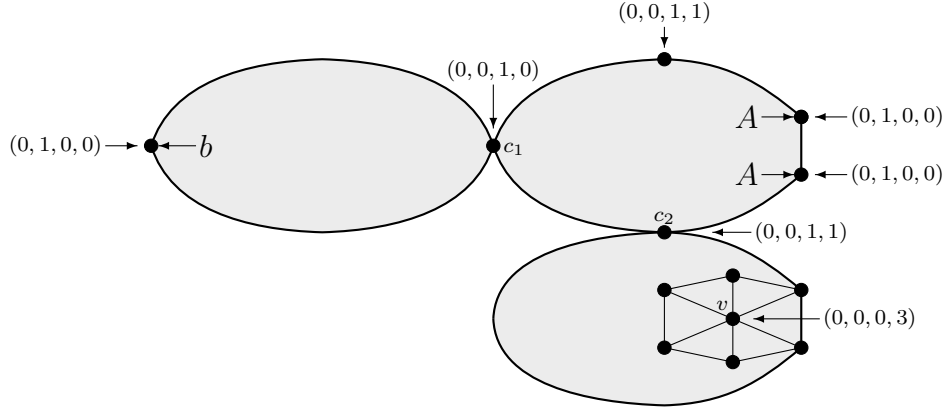

\begin{lemma}\label{lem:induction}
  Painter has a winning strategy in the $(G,A,b)$-refined game.
\end{lemma}

Before the proof, we show how to use Lemma~\ref{lem:induction} to prove Theorem~\ref{thm:positive}.

\begin{proof}[Proof of Theorem~\ref{thm:positive}]
  Suppose to the contrary that a planar graph $G$ is not 3-defective 3-paintable.
  Adding some edges to $G$ introduce additional constraints for Painter in the 3-defective 3-painting game.
  Thus, we can assume that $G$ is connected.
  Choose any plane embedding of $G$.
  Choose any $b$ on the boundary of the outer face.
  By Lemma~\ref{lem:induction}, Painter has a winning strategy $S$ in the $(G,\emptyset,b)$-refined game.
  The strategy $S$ is a valid winning strategy in the 3-defective 3-painting game on $G$.
\end{proof}

Before we present the proof of Lemma~\ref{lem:induction}, we briefly introduce some techniques that we frequently use in the proof.

Assume that Painter has a winning strategy $S_1$ in the $(G_1,A_1,b_1)$-refined game $\Gamma_1$.
Now, if we modify the initial state of the game by adding some more tokens, or increasing value of some tokens, then obviously Painter has a winning strategy in the resulting game.
Thus, in the proof of Lemma~\ref{lem:induction} when some vertex has too many tokens, or has tokens of too great value, we can \emph{devalue} the token function and use the winning strategy $S_1$.
We say that a token function $g$ is \emph{sufficient} for $\Gamma_1$ if it is equal to or can be devalued to the token function in $\Gamma_1$.

In order to find a winning strategy for Painter in the $(G,A,b)$-refined game $\Gamma$, we often divide the graph $G$ into $k$, possibly overlapping, parts $G_1 = G\sbrac{V_1},\ldots,G_k = G\sbrac{V_k}$ and consider $(G_i,A_i,b_i)$-refined games.
The division of the graph and choice of special vertices $A_1,b_1,\ldots,A_k,b_k$ depends on the structure of $G$.
Then, we can use induction and assume that Painter has a winning strategy $S_i$ in each $(G_i,A_i,b_i)$-refined game $\Gamma_i$.
We present the following \emph{composed} strategy $S$ in $\Gamma$ that uses strategies $S_1,\ldots,S_k$ sequentially.

%Let $1 \leq i < j \leq k$.
%If $V_i$ and $V_j$ are disjoint and there is no edge connecting a vertex in $G_i$ with a vertex in $G_j$, then Painter can use $S_i$ and $S_j$ independently -- coloring a vertex in $G_i$ does not give any defect to vertices in $G_j$.
%When there are edges connecting $V_i$ with $V_j$ then we must take them into account when calculating the defect that any single vertex receives.
%Usually, there will be only a few such edges or we will be able to argue that the end-points of such edges are not colored in the same round.

For a vertex $v$, let $i_v$ be the first index $i$ such that $v \in V_i$.
Strategy $S$ will use strategy $S_{i_v}$ to decide whether $v$ gets colored.
If $v \in V_j$ for some $j > i_v$ then we will have that $v = b_j$ or $v \in A_j$.
This way we get that vertex $v$ has only one token and will be marked only once in game $\Gamma_j$ -- in the round $v$ gets colored in $S_{i_v}$.

Now, we introduce a very useful technique. %that will allow us to assume that Lister never marks some pairs of vertices simultaneously.
For a vertex $v$, let \emph{slack} of $v$ be the highest number $s$ such that we can remove $s$ most valuable tokens from $v$ and the resulting token function is sufficient for $\Gamma_{i_v}$.
The slack of any vertex is at most $2$.
Now, let $U(v)$ be some carefully selected set of neighbors of $v$ in $G$.
We say that $v$ \emph{gives away a token} to each $u$ in $U(v)$ to describe the following behavior.
Assume that the size of $U(v)$ does not exceed the slack of $v$, and, for each $u$ in $U(v)$, either $u$ has only one token in $\Gamma$ or $i_u < i_v$.
In particular, in every round, when we use strategy $S_{i_v}$ to decide whether or not to color $v$, we already know if any vertex in $U(v)$ will be colored in this round.
We say that $v$ is \emph{blocked} in some round by $u \in U(v)$ if $u$ is colored in this round, \ie{} either $u$ has only one token and $u$ is marked, or $i_u < i_v$ and $S_{i_u}$ colored $u$.
When vertex $v$ is blocked in some round then we will not mark it in $\Gamma_{i_v}$.
So, vertex $v$ will be marked in game $\Gamma_{i_v}$ possibly fewer times than it is marked in game $\Gamma$.
Each vertex $u \in U(v)$ blocks $v$ at most once during the game, and the number of times vertex $v$ is blocked will not exceed the slack of $v$.

%Now, we introduce a very useful technique that we use in designing Painter and Lister strategies in $S_1,\ldots,S_k$.
%Suppose $u,v$ is are vertices of $G$ such that $u$ has only one token in $\Gamma$ or $i_u < i_v$.
%In particular, in every round, when Painter of $\Gamma_{v_j}$ considers whether or not to color $v$,
%it already knows whether $u$ has been colored in that round.
%We say that $v$ \emph{gives away} a token to $u$ if Painter's strategy of $\Gamma_{v_j}$ never colors $v$ when $u$ is colored.
%We model this strategy by not marking $v$ in $\Gamma_j$ in in the round when $u$ is colored.
%So, in the game $\Gamma_j$ the vertex $v$ can be marked with at least one token less (possibly of the greatest value) than in the game $\Gamma$.
%Assuming the worst case (the vertex $v$ is marked with the token of the greatest value when $u$ is colored),
%we set the initial list of tokens for $v$ in $\Gamma_j$ by deleting the most valuable token from its initial list of tokens in $\Gamma$.
%Clearly, if $v$ gives away a token to two vertices, we update its initial token list in $\Gamma_{v_j}$ by deleting two the most valuable tokens.

Let $M$ be a set of vertices marked by Lister in some round.
For $i=1,\ldots,k$, Painter constructs the set $M_i$ and uses strategy $S_i$ to find a response $X_i$ for move $M_i$ in the game $\Gamma_i$.
The set $M_i$ depends on the responses given by strategies $S_1,\ldots,S_{i-1}$ and is defined as
  $$M_i=\set{v \in M \cap V_i: v \text{ is not blocked, } i = i_v \text{ or } i > i_v \text{ and } v \in X_{i_v}}\text{.}$$
Additionally, we need to decide the value of the token removed from each marked vertex.
Observe that the regular boundary vertices are the only vertices that have tokens of distinct values, \ie{} one token of value $2$ and one token of value $3$.
Usually, this will be a natural and simple decision.
In many cases, we will simply use the same value as Lister chose in $\Gamma$.
Nevertheless, in some scenarios, we will have to be more careful about this choice.
Details will be presented when needed.

Strategy $S$ colors the set $X = \set{v \in M: v \in X_{i_v}}$.
In order to prove that the composed strategy $S$ is a winning strategy in the game $\Gamma$ we need to argue that:
\begin{itemize}
  \item The token function in $\Gamma$ after removal of tokens that were given away is sufficient for each $\Gamma_i$.
    This is an easy calculation and we will omit it in most of the cases.
  \item The defects that any single vertex receives in games $\Gamma_1,\ldots,\Gamma_k$ do not exceed the value of a token removed by Lister in game $\Gamma$.
    This will usually be the most important argument.
  \item If $a \in A$ is a $C$-neighbor of $b$, then $\theta(a) \subseteq \set{b}$.
  \item In each round $\norm{M_i \cap A_i} \leq 1$.
In order to guarantee this, we will have that if some $A_i$ has two elements, then either $A_i = A$, or one of the vertices in $A_i$ gave away a token to the other.
Observe that if some vertex $v$ and $u \in U(v)$ are vertices in $G_i$, then they are never both marked in the same round in the game $\Gamma_i$.
\end{itemize}

In figures that present game divisions we use the following schemas:
\begin{itemize}
  \item Vertices of $G_1,\ldots,G_k$ lie inside or on the boundary of regions filled with different shades of gray.
  \item We denote $A_i$, and $b_i$ with $A$ and $b$ inside the region corresponding to $G_i$.
  \item We draw an edge directed from $v$ to $u$ to mark that $v$ gives away a token to $u$.
\end{itemize}

\begin{proof}[Proof of Lemma~\ref{lem:induction}]
  We prove the lemma by induction.
  Assume, that $G$ is the smallest, in terms of the number of vertices, connected plane graph for which the lemma does not hold.
  Assume, that all internal faces of $G$ are triangulated, as adding edges that do not change the boundary walk introduce only additional constraints for Painter.
  Let $C$ be the boundary walk of the outer face of $G$, and $A$ and $b$ be the special vertices.
  Any closed walk $W$ in $G$ divides the plane into connected regions.
  Let $\inter{W}$ denote the subgraph of $G$ induced by the vertices that are inside the closure of bounded connected regions of the plane with edges of $W$ removed.
  For a simple path $P$ and two distinct vertices $u$, $v$ on that path, let $P[u,v]$ denote the subpath of $P$ that traverses $P$ from vertex $u$ to vertex $v$.
  Similarly, for a simple cycle $D$ in $G$ and two distinct vertices $u$, $v$ on that cycle, let $D[u,v]$ denote the subpath of $D$ that traverses $D$ in the clockwise direction from vertex $u$ to vertex $v$.
  For a path $Q[u,v]$, we use notation $Q(u,v)$, $Q[u,v)$, and $Q(u,v]$ to denote $Q[u,v] \grminus \set{u,v}$, $Q(u,v) + u$, and $Q(u,v) + v$ respectively.

  The proof divides into several cases.
  The analysis of Case~\ref{case:base} is the basis of the induction and shows that $G$ has at least four vertices.
  The analysis of Cases~\ref{case:bridge} and~\ref{case:cut} shows that $G$ is biconnected.
  The analysis of Cases~\ref{case:triangle} and~\ref{case:ab} shows that vertex $b$ is not adjacent to vertices in $A$.
  Case~\ref{case:final} is the final case of the induction and shows that $G$ does not exist.

\Cases

\Case{$G$ has at most three vertices}\label{case:base}
  Observe that each vertex has a token of value at least $1$.
  If there are at most two vertices in $G$, then all vertices can be colored simultaneously in the same round.

  Now, assume that $G$ has exactly three vertices.
  Observe that all vertices are in $C$.
  If $A$ is empty, choose any vertex $x$ other than $b$, devalue token function for $x$ and set $A=\set{x}$.
  The winning strategy in the resulting game is also a winning strategy in the original game.

  If $A$ has exactly one element $a$, let $x$ be the third vertex other than $a$ and $b$.
  If $x$ is an $(A,b)$-cut, then all three vertices can be colored simultaneously in the same round.
  If $x$ is adjacent to $a$, but not an $(A,b)$-cut, then $x$ has two tokens, $x$ gives away a token to $a$, devalue token function for $x$ and set $A=\set{a,x}$.
  The winning strategy in the resulting game is also a winning strategy in the original game.

  If $x$ is not adjacent to $a$, and not an $(A,b)$-cut, then $x$ has two tokens.
  We add the edge $\set{a,x}$ to the graph.
  Vertex $x$ gives away a token to $a$ and set $A=\set{a,x}$.

  If $A$ has two elements, then both elements of $A$ are not marked in the same round.
  Thus, Painter can color each vertex $v$ in the first round that $v$ is marked in.

\Case{$G$ has a bridge}\label{case:bridge}
  Let edge $e=\set{x,y}$ be a bridge in $G$.
  Let $G_1$ and $G_2$ be the two connected components of $G \grminus e$ with $x$ in $G_1$, and $y$ in $G_2$.
  Without loss of generality, assume that the special vertex $b$ is in $G_1$.
  We divide this case depending on the position of $A$ relative to $e$.

\Subcase{$A \subset G_1$}\label{case:bridge_G1}
  In this case, vertex $y$ is not an $(A,b)$-cut and has two tokens.
  We divide the game into smaller games:
  \begin{itemize}
    \item $\Gamma_1=(G_1, A, b)$,
    \item $\Gamma_2=(G_2, \emptyset, y)$.
  \end{itemize}
  Vertex $y$ gives away a token to $x$.
  As a result, we have that $x$ gets defect only in $\Gamma_1$, and $y$ gets defect only in $\Gamma_2$.
  If $a \in A$ is a $C$-neighbor of $b$ then game $\Gamma_1$ ensures that $\theta(a) \subseteq \set{b}$.

\Subcase{$A \cap G_1 \neq \emptyset$, and $A \cap G_2 \neq \emptyset$}\label{case:bridge_G1G2}
  In this case we have that $A = \set{x,y}$.
  We divide the game into smaller games:
  \begin{itemize}
    \item $\Gamma_1=(G_1, \set{x}, b)$,
    \item $\Gamma_2=(G_2, \emptyset, y)$.
  \end{itemize}
  As Lister is not allowed to mark both $x$ and $y$ in the same round, vertex $x$ gets defect only in $\Gamma_1$, and vertex $y$ gets defect only in $\Gamma_2$.
  Vertex $y$ is not adjacent to $b$, and if $x$ is a $C$-neighbor of $b$ then game $\Gamma_1$ ensures that $\theta(x) \subseteq \set{b}$.

\Subcase{$A \subseteq G_2$}\label{case:bridge_G2}
  Vertex $x$ is either an $(A,b)$-cut or $x=b$, and similarly vertex $y$ is either an $(A,b)$-cut or $y \in A$.
  We divide this case further depending on the size of $G_1$ and these possibilities.

  \Subsubcase{$\norm{G_1} \geq 2$, $x$ is an $(A,b)$-cut}\label{case:bridge_G2_bigcut}
  In this case, vertex $x$ has a single token of value $2$, and vertices in $A$ are not adjacent to $b$.
  We divide the game into smaller games:
  \begin{itemize}
    \item $\Gamma_1=(G_1, \set{x}, b)$,
    \item $\Gamma_2=(G_2+x, A, x)$.
  \end{itemize}
  As a result, vertex $x$ gets at most one defect in $\Gamma_1$, and at most one defect in $\Gamma_2$.

\Subsubcase{$\norm{G_1} \geq 2$, $x=b$}\label{case:bridge_G2_bigb}
  Let $z$ be a $C$-neighbor of $x$ in $G_1$.
  Vertex $z$ is not an $(A,b)$-cut and has two tokens.
  We divide the game into smaller games:
  \begin{itemize}
    \item $\Gamma_1=(G_1, \set{b}, z)$,
    \item $\Gamma_2=(G_2+b, A, b)$.
  \end{itemize}
  Vertex $z$ gives away a token to $b$.
  As a result, vertex $b$ gets at most one defect in $\Gamma_2$ and no defect in $\Gamma_1$ -- rules of the game $\Gamma_1$ enforce $\theta(b) \subseteq \set{z}$ and $z$ gave away a token to $b$.
  If $y \in A$ then $y$ is a $C$-neighbor of $b$ and game $\Gamma_2$ ensures that $\theta(y) \subseteq \set{b}$.
  Vertices in $A$ other than $y$ are not adjacent to $b$.

\Subsubcase{$\norm{G_1} = 1$, $y$ is an $(A,b)$-cut}\label{case:bridge_G2_smallcut}
  We observe that $\norm{G_1}=1$ implies $x=b$ and divide the game into smaller games:
  \begin{itemize}
    \item $\Gamma_1=(G_1, \emptyset, x)$,
    \item $\Gamma_2=(G_2, A, y)$.
  \end{itemize}
  Vertex $x$ obviously gets at most one defect.
  Vertex $y$ gets at most one defect from $x$ and at most one defect in $\Gamma_2$.

\Subsubcase{$\norm{G_1} = 1$, $y \in A$, $y$ has at least two neighbors in $G_2$}\label{case:bridge_G2_smalla2}
  In this case, vertex $y$ has at least two $C$-neighbors in $G_2$.
  Choose vertex $z$, a $C$-neighbor of $y$ in $G_2$ that is not in $A$.
  Vertex $z$ is not an $(A,b)$-cut and has two tokens.
  We divide the game into smaller games:
  \begin{itemize}
    \item $\Gamma_1=(G_1, \emptyset, x)$,
    \item $\Gamma_2=(G_2, A, z)$.
  \end{itemize}
  Vertex $z$ gives away a token to $y$.
  Vertex $x$ obviously gets at most one defect.
  Vertex $y$ gets at most one defect from $x$ and no defect in $\Gamma_2$ -- rules of the game $\Gamma_2$ enforce $\theta(y) \subseteq \set{z}$ and $z$ gave away a token to $y$.

\Subsubcase{$\norm{G_1} = 1$, $y \in A$, $y$ has only one neighbor in $G_2$}\label{case:bridge_G2_smalla1}
  Let $z$ be the only neighbor of $y$ in $G_2$.
  In this case, we apply Case~\ref{case:bridge_G1G2} if $z \in A$, or Case~\ref{case:bridge_G1} if $z \notin A$ of the induction for the bridge $\set{y,z}$.

\medskip

In the analysis of the following cases we assume that each vertex has degree at least two.
Indeed, a vertex of degree one is incident to a bridge.
For each vertex not in $C$, the neighbors of $v$ traversed clockwise induce a simple cycle in $G$.
Let $NC(v)$ denote this cycle.
Furthermore, we assume that $A$ has exactly two elements, say $a_1$ and $a_2$.
If $A = \emptyset$, choose any vertex $a_1$ in $C$ other than $b$ and set $A=\set{a_1}$.
If $A = \set{a_1}$, choose a vertex $a_2$, a $C$-neighbor of $a_1$ other than $b$ and not an $(\set{a_1},b)$-cut.
Vertex $a_2$ gives away a token to $a_1$ and set $A=\set{a_1,a_2}$.

\begin{figure}[h]

\begin{tikzpicture}[scale=0.5]

\tikzset{->-/.style={decoration={
  markings,
  mark=at position #1 with {\arrow{latex}}},postaction={decorate}}}

\begin{scope}[shift={(-6.5,0)}]

\coordinate (w) at (0,0) {};
\coordinate (g11) at (-4,-2.5) {};
\coordinate (g11m) at (-2,-1.25) {};
\coordinate (g12) at (-4,2.5) {};
\coordinate (g12m) at (-2,1.25) {};

\coordinate (g21) at (4,2.5) {};
\coordinate (g22) at (4,-2.5) {};
\coordinate (g21m) at (2,1.25) {};
\coordinate (g22m) at (2,-1.25) {};

\coordinate (g31) at (2.5,-4) {};
\coordinate (g32) at (-2.5,-4) {};
\coordinate (g31m) at (1.25,-2) {};
\coordinate (g11m) at (-1.25,-2) {};

\filldraw[fill=gray!10, draw=none]
(w) to (g11) to[out=-180,in=180] (g12) -- cycle;

\filldraw[fill=gray!25, draw=none]
(w) to (g21) to[out=0,in=0] (g22) -- cycle;

\filldraw[fill=gray!45, draw=none]
(w) to (g31) to[out=-90,in=-90] (g32) -- cycle;

\draw[thick] (w) to (g11) to[out=-180,in=180] (g12) -- cycle;
\draw[thick] (w) to (g21) to[out=0,in=0] (g22) -- cycle;
\draw[thick, ->-=0.6] (g21m) to (w);
\draw[thick] (w) to (g31) to[out=-90,in=-90] (g32) -- cycle;
\draw[thick, ->-=0.6] (g31m) to (w);

\begin{tiny}
\tikzstyle{every node}=[circle,minimum size=5pt,inner sep=0pt,draw,fill]
\node (wn) at (w) {};
\node (y2n) at (g21m) {};
\node (y3n) at (g31m) {};
\node (A1n) at (g12m) {};
\node (A2n) at (g12) {};
\node (bn) at (g11) {};

\tikzstyle{every node}=[inner sep=0pt]
\node at ($(w) + (0.0,0.5)$) {$w$};
\node at ($(g21m) + (0.0,0.6)$) {$y_2$};
\node at ($(g31m) + (0.75,0.0)$) {$y_3$};
\node at ($(g12m) + (0.0,0.6)$) {$a_1$};
\node at ($(g12) + (0.0,0.6)$) {$a_2$};
\node at ($(g11) + (0.0,-0.65)$) {$b$};
\end{tiny}

%Game1
\tikzstyle{every node}=[inner sep=1pt]
\node at (-4.1,-0.0) {$\Gamma_1$};
\begin{tiny}
\node (G1A1) at ($(g12m) + (0.45,-0.95)$) {$A$};
\node (G1A2) at ($(g12) + (-0.45,-0.85)$) {$A$};
\node (G1b) at ($(g11) + (-0.4, 0.8)$) {$b$};
\draw[->-=1] ($(A2n) + (0.0,-1.3)$) -- (A2n);
\draw[->-=1] ($(A1n) + (0.0,-1.3)$) -- (A1n);
\draw[->-=1] ($(bn) + (0.0,1.3)$) -- (bn);
\end{tiny}

%Game2
\tikzstyle{every node}=[inner sep=1pt]
\node at (4.3,-0.0) {$\Gamma_2$};
\begin{tiny}
\node (G2A) at ($(wn) + (1.4,0)$) {$A$};
\node (G2b) at ($(y2n) + (0.4,-0.8)$) {$b$};
\draw[->-=1] (G2A) -- (wn);
\draw[->-=1] ($(y2n) + (0.0,-1.2)$) -- (y2n);
\end{tiny}

%Game3
\tikzstyle{every node}=[inner sep=1pt]
\node at (0.1,-4.3) {$\Gamma_3$};
\begin{tiny}
\node (G3A) at ($(wn) + (0,-1.4)$) {$A$};
\node (G3b) at ($(y3n) + (-0.6,-0.6)$) {$b$};
%\draw[->-=1] ($(y2n) + (0.0,-1.2)$) -- (y2n);
\draw[->-=1] (G3A) -- (wn);
\draw[->-=1] ($(y3n) + (-0.9,0)$) -- (y3n);
\end{tiny}
\end{scope}

\begin{scope}[shift={(6.5,0)}]

\coordinate (w) at (0,0) {};
\coordinate (g11) at (-4,-2.5) {};
\coordinate (g11m) at (-2,-1.25) {};
\coordinate (g12) at (-4,2.5) {};
\coordinate (g12m) at (-2,1.25) {};

\coordinate (g21) at (4,2.5) {};
\coordinate (g22) at (4,-2.5) {};
\coordinate (g21m) at (2,1.25) {};
\coordinate (g22m) at (2,-1.25) {};

\coordinate (g31) at (2.5,-4) {};
\coordinate (g32) at (-2.5,-4) {};
\coordinate (g31m) at (1.25,-2) {};
\coordinate (g11m) at (-1.25,-2) {};

\filldraw[fill=gray!10, draw=none]
(w) to (g11) to[out=-180,in=180] (g12) -- cycle;

\filldraw[fill=gray!25, draw=none]
(w) to (g21) to[out=0,in=0] (g22) -- cycle;

\filldraw[fill=gray!45, draw=none]
(w) to (g31) to[out=-90,in=-90] (g32) -- cycle;

\draw[thick] (w) to (g11) to[out=-180,in=180] (g12) -- cycle;
\draw[thick] (w) to (g21) to[out=0,in=0] (g22) -- cycle;
%\draw[thick, ->-=0.6] (g21m) to (w);
\draw[thick] (w) to (g31) to[out=-90,in=-90] (g32) -- cycle;
\draw[thick, ->-=0.6] (g31m) to (w);

\begin{tiny}
\tikzstyle{every node}=[circle,minimum size=5pt,inner sep=0pt,draw,fill]
\node (wn) at (w) {};
\node (y2n) at (g21m) {};
\node (y3n) at (g31m) {};
\node (bn) at (g11) {};

\tikzstyle{every node}=[inner sep=0pt]
\node at ($(w) + (0.0,0.6)$) {$a_1$};
\node at ($(g21m) + (0.0,0.6)$) {$a_2$};
\node at ($(g31m) + (0.75,0.0)$) {$y_3$};
\node at ($(g11) + (0.0,-0.6)$) {$b$};
\end{tiny}

%Game1
\tikzstyle{every node}=[inner sep=1pt]
\node at (-4.1,-0.0) {$\Gamma_1$};
\begin{tiny}
\node (G1A) at ($(wn) + (-1.4,0)$) {$A$};
\node (G1b) at ($(g11) + (-0.4, 0.8)$) {$b$};
\draw[->-=1] (G1A) -- (wn);
\draw[->-=1] ($(bn) + (0.0,1.3)$) -- (bn);
\end{tiny}

%Game2
\tikzstyle{every node}=[inner sep=1pt]
\node at (4.3,-0.0) {$\Gamma_2$};
\begin{tiny}
\node (G2A) at ($(wn) + (1.4,0)$) {$A$};
\node (G2b) at ($(y2n) + (0.4,-0.7)$) {$b$};
\draw[->-=1] (G2A) -- (wn);
\draw[->-=1] ($(y2n)+(0.0,-1.1)$) -- (y2n);
\end{tiny}

%Game3
\tikzstyle{every node}=[inner sep=1pt]
\node at (0.1,-4.3) {$\Gamma_3$};
\begin{tiny}
\node (G3A) at ($(wn) + (0,-1.4)$) {$A$};
\node (G3b) at ($(y3n) + (-0.6,-0.6)$) {$b$};
%\draw[->-=1] ($(y2n) + (0.0,-1.2)$) -- (y2n);
\draw[->-=1] (G3A) -- (wn);
\draw[->-=1] ($(y3n) + (-0.9,0)$) -- (y3n);
\end{tiny}

\end{scope}

\begin{scope}[shift={(0,-10)}]

\coordinate (w) at (0,0) {};
\coordinate (g11) at (-4,-2.5) {};
\coordinate (g11m) at (-2,-1.25) {};
\coordinate (g12) at (-4,2.5) {};
\coordinate (g12m) at (-2,1.25) {};

\coordinate (g21) at (4,2.5) {};
\coordinate (g22) at (4,-2.5) {};
\coordinate (g21m) at (2,1.25) {};
\coordinate (g22m) at (2,-1.25) {};

\coordinate (g31) at (2.5,-4) {};
\coordinate (g32) at (-2.5,-4) {};
\coordinate (g31m) at (1.25,-2) {};
\coordinate (g11m) at (-1.25,-2) {};

\filldraw[fill=gray!10, draw=none]
(w) to (g11) to[out=-180,in=180] (g12) -- cycle;

\filldraw[fill=gray!25, draw=none]
(w) to (g21) to[out=0,in=0] (g22) -- cycle;

\filldraw[fill=gray!45, draw=none]
(w) to (g31) to[out=-90,in=-90] (g32) -- cycle;

\draw[thick] (w) to (g11) to[out=-180,in=180] (g12) -- cycle;
\draw[thick] (w) to (g21) to[out=0,in=0] (g22) -- cycle;
\draw[thick] (w) to (g31) to[out=-90,in=-90] (g32) -- cycle;
\draw[thick, ->-=0.6] (g31m) to (w);

\begin{tiny}
\tikzstyle{every node}=[circle,minimum size=5pt,inner sep=0pt,draw,fill]
\node (wn) at (w) {};
\node (a1n) at (g21m) {};
\node (a2n) at (g21) {};
\node (y3n) at (g31m) {};
\node (bn) at (g11) {};

\tikzstyle{every node}=[inner sep=0pt]
\node at ($(w) + (0.0,0.6)$) {$w$};
\node at ($(g21m) + (0.0,0.6)$) {$a_1$};
\node at ($(g21) + (0.0,0.6)$) {$a_2$};
\node at ($(g31m) + (0.75,0.0)$) {$y_3$};
\node at ($(g11) + (0.0,-0.6)$) {$b$};
\end{tiny}

%Game1
\tikzstyle{every node}=[inner sep=1pt]
\node at (-4.1,-0.0) {$\Gamma_1$};
\begin{tiny}
\node (G1A) at ($(wn) + (-1.4,0)$) {$A$};
\node (G1b) at ($(g11) + (-0.4, 0.8)$) {$b$};
\draw[->-=1] (G1A) -- (wn);
\draw[->-=1] ($(bn) + (0.0,1.3)$) -- (bn);
\end{tiny}

%Game2
\tikzstyle{every node}=[inner sep=1pt]
\node at (4.3,-0.0) {$\Gamma_2$};
\begin{tiny}
\node (G2b) at ($(wn) + (1.4,0)$) {$b$};
\node (G2A1) at ($(a1n) + (0.5,-0.7)$) {$A$};
\node (G2A1) at ($(a2n) + (-0.5,-0.8)$) {$A$};
\draw[->-=1] (G2b) -- (wn);
\draw[->-=1] ($(a1n)+(0.0,-1.1)$) -- (a1n);
\draw[->-=1] ($(a2n)+(0.0,-1.1)$) -- (a2n);
\end{tiny}

%Game3
\tikzstyle{every node}=[inner sep=1pt]
\node at (0.1,-4.3) {$\Gamma_3$};
\begin{tiny}
\node (G3A) at ($(wn) + (0,-1.4)$) {$A$};
\node (G3b) at ($(y3n) + (-0.6,-0.6)$) {$b$};
%\draw[->-=1] ($(y2n) + (0.0,-1.2)$) -- (y2n);
\draw[->-=1] (G3A) -- (wn);
\draw[->-=1] ($(y3n) + (-0.9,0)$) -- (y3n);
\end{tiny}

\end{scope}

\end{tikzpicture}

\caption{
Game division in Cases~\ref{case:cut_G1}, \ref{case:cut_G2A}, and \ref{case:cut_G2noA}, respectively.
}
\label{fig:cut}
\end{figure}
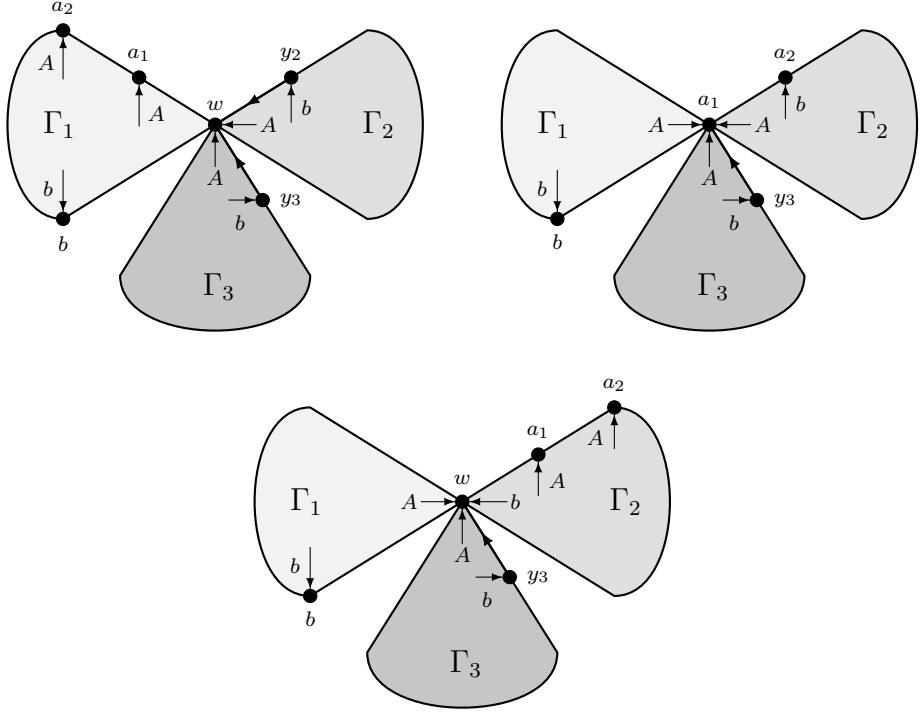

\Case{$G$ has a cut-point}\label{case:cut}
  Let vertex $w$ be a cut-point in $G$.
  Let $G_1,\ldots,G_k$ be the components of $G \grminus w$.
  Update each graph $G_i$ by adding vertex $w$ back to it.
  Without loss of generality, assume that $b$ is in $G_1$ and that $A$ is contained either in $G_1$, or in $G_2$.
  Let $y_i$, for $i=1\ldots,k$, be any $C$-neighbor of $w$ in $G_i$.
  We divide this case depending on the position of $A$ relative to $w$.
  Figure~\ref{fig:cut} depicts the game divisions that we use in the subcases.

\Subcase{$A \subset G_1$}\label{case:cut_G1}
  We divide the game into smaller games:
  \begin{itemize}
    \item $\Gamma_1=(G_1,A,b)$,
    \item $\Gamma_i=(G_i,\set{w},y_i)$, for $i=2,\ldots,k$.
  \end{itemize}
  Each vertex $y_i$, for $i=2,\ldots,k$ gives away a token to vertex $w$.
  As a result, vertex $w$ gets no defect in the games $\Gamma_2,\ldots,\Gamma_k$. % as rules of the special painting game $\Gamma_i$ allow vertex $w$ to get defect only from $y_i$.
  If $a \in A$ is a $C$-neighbor of $b$, then game $\Gamma_1$ ensures that $\theta(a) \subseteq \set{b}$.

\Subcase{$A \subset G_2$, $w \in A$}\label{case:cut_G2A}
  Without loss of generality, $w=a_1$.
  We divide the game into smaller games:
  \begin{itemize}
    \item $\Gamma_1=(G_1,\set{a_1},b)$,
    \item $\Gamma_2=(G_2, \set{a_1}, a_2)$,
    \item $\Gamma_i=(G_i,\set{w},y_i)$, for $i=3,\ldots,k$.
  \end{itemize}
  Each vertex $y_i$, for $i=3,\ldots,k$ gives away a token to vertex $w$.
  Vertices $w=a_1$ and $a_2$ are not marked in the same round.
  As a result, vertex $w$ gets no defect in the games $\Gamma_2,\ldots,\Gamma_k$. % as rules of the special painting game $\Gamma_i$ allow vertex $w$ to get defect only from $y_i$.
  If $a_1$ is a $C$-neighbor of $b$, then game $\Gamma_1$ ensures that $\theta(a_1) \subseteq \set{b}$.
  Vertex $a_2$ is not adjacent to $b$.

\Subcase{$A \subset G_2$, $w \notin A$}\label{case:cut_G2noA}
  In this case, vertex $w$ is an $(A,b)$-cut.
  We divide the game into smaller games:
  \begin{itemize}
    \item $\Gamma_1=(G_1,\set{w},b)$,
    \item $\Gamma_2=(G_2,A,w)$,
    \item $\Gamma_i=(G_i,\set{w},y_i)$, for $i=3,\ldots,k$.
  \end{itemize}
  Each vertex $y_i$, for $i=3,\ldots,k$ gives away a token to vertex $w$.
  Vertex $w$ gets at most one defect in each of the games $\Gamma_1$, $\Gamma_2$ and no defect in the games $\Gamma_3,\ldots,\Gamma_k$. % as rules of the special painting game $\Gamma_i$ allow vertex $w$ to get defect only from $y_i$.
  Vertices in $A$ are not adjacent to $b$.

\medskip

In the analysis of the following cases we assume that $G$ is biconnected.
Thus, the boundary walk $C$ is a simple cycle.
For a vertex $v$ in $C$ we define $v^+$, and $v^-$ to be respectively the next, and the previous vertex in $C$ when $C$ is traversed clockwise.
We define the path $NP(v)$ that traverses neighbors of $v$ clockwise from $v^+$ to $v^-$.
For any two vertices $u$ and $v$ in $C$, let $N(u,v)$ denote the set of common neighbors of $u$ and $v$, \ie{} $N(u) \cap N(v)$.
Now, assume $u$ and $v$ are $C$-neighbors.
The \emph{minimum common neighbor} of $u$ and $v$, denoted $\minn(u,v)$, is a vertex $w$ in $N(u,v)$ such that $\inter{u,v,w,u}$ contains no other common neighbor of $u$ and $v$.
The \emph{maximum common neighbor} of $u$ and $v$, denoted $\maxn(u,v)$, is a vertex $w$ in $N(u,v)$ such that $\inter{u,v,w,u}$ contains all other common neighbors of $u$ and $v$.
As $u$ and $v$ are $C$-neighbors, any two common neighbors $x_1$, $x_2$ of $u$ and $v$ are on the same side of the edge $\set{u,v}$.
Thus, we have that one of the sets $\inter{x_1,u,v,x_1}$, $\inter{x_2,u,v,x_2}$ is contained in the other and that both $\minn(u,v)$ and $\maxn(u,v)$ exist.
Let $a_1$ and $a_2$ be the elements of $A$ so that $a_1$, $a_2$, $b$ appear in this order when $C$ is traversed clockwise.

\Case{$C$ is a triangle}\label{case:triangle}
  We divide this case depending on the existence of a common neighbor of $a_1$, $a_2$, and $b$.

\begin{figure}[h]
\centering
\begin{tikzpicture}[xscale=0.9, yscale=0.6]

\tikzset{->-/.style={decoration={
  markings,
  mark=at position #1 with {\arrow{latex}}},postaction={decorate}}}

\begin{scope}

\coordinate (a2) at (-5,0) {};
\coordinate (a1) at (5,0) {};
\coordinate (b) at (0,10) {};
\coordinate (d) at (0,3) {};

\coordinate (l1) at (-2,3.3) {};
\coordinate (l2) at (-2,5) {};
\coordinate (r1) at (2,3.3) {};
\coordinate (r2) at (2,5) {};

%Gamma1
\filldraw[fill=gray!10, draw=none]
(d) to[out=0,in=-130] (r1) to[out=70,in=-70] (r2) to[out=110,in=-70] (b) -- cycle;

%Gamma2
\filldraw[fill=gray!25, draw=none]
  (d) to[out=180,in=-50] (l1) to[out=110,in=-110] (l2) to[out=70,in=250] (b) -- cycle;

%Gamma3
\filldraw[fill=gray!45, draw=none]
  (a2) -- (d) -- (a1) -- cycle;

\draw[thick] (a1) -- (a2);
\draw[thick] (a1) -- (b);
\draw[thick] (a2) -- (b);

\draw[thick,->-=.5] (d) to (a1);
\draw[thick,->-=.5] (d) to (a2);
\draw[thick] (b) -- (d);
\draw[thick] (d) to[out=180,in=-50] (l1);
\draw[thick] (l1) to[out=110,in=-110] (l2);
\draw[thick] (l2) to[out=70,in=250] (b);

\draw[thick] (d) to[out=0,in=-130] (r1);
\draw[thick] (r1) to[out=70,in=-70] (r2);
\draw[thick] (r2) to[out=110,in=-70] (b);

\draw[->-=.4] (l1) -- (a2);
\draw[->-=.4] (l2) -- (a2);
\draw[->-=.4] (r1) -- (a1);
\draw[->-=.4] (r2) -- (a1);

\tikzstyle{every node}=[circle,minimum size=5pt,inner sep=0pt,draw,fill]
\node (a1n) at (a1) {};
\node (a2n) at (a2) {};
\node (bn) at (b) {};

\tikzstyle{every node}=[circle,minimum size=11pt,inner sep=0pt,draw,fill=white]
\begin{tiny}
\node (dn) at (0,3) {$d$};
\end{tiny}

\begin{tiny}
\tikzstyle{every node}=[inner sep=0pt]
\node at ($(a1) + (0.3,-0.3)$) {$a_1$};
\node at ($(a2) + (-0.3,-0.3)$) {$a_2$};
\node at ($(b) + (0,0.4)$) {$b$};
\end{tiny}

%GAMMA1
\tikzstyle{every node}=[inner sep=1pt]
\node at (1.05,5) {$\Gamma_1$};
\begin{tiny}
\node at (-1.8,3.9) {$P_2$};
\node (G1A) at ($(b) + (0.2,-1.3)$) {$A$};
\node (G1b) at ($(d) + (0.45, 0.8)$) {$b$};
\end{tiny}
\draw[->-=0.7] (G1A) -- (bn);
\draw[->-=1] (G1b) -- (dn);

%GAMMA2
\tikzstyle{every node}=[inner sep=1pt]
\node at (-1,5) {$\Gamma_2$};
\begin{tiny}
\node at (1.8,3.9) {$P_1$};
\node (G2A) at ($(b) + (-0.2,-1.3)$) {$A$};
\node (G2b) at ($(d) + (-0.45, 0.8)$) {$b$};
\end{tiny}
\draw[->-=0.7] (G2A) -- (bn);
\draw[->-=1] (G2b) -- (dn);

%GAMMA3
\tikzstyle{every node}=[inner sep=1pt]
\node at (0.1,0.9) {$\Gamma_3$};
\begin{tiny}
\node (G3A1) at ($(a2) + (1.1,0.3)$) {$A$};
\node (G3A2) at ($(a1) + (-1.1,0.3)$) {$A$};
\node (G3b) at ($(d) + (0, -1.0)$) {$b$};
\end{tiny}
\draw[->-=0.5] (G3A1) -- (a2n);
\draw[->-=0.5] (G3A2) -- (a1n);
\draw[->-=1] (G3b) -- (dn);

\end{scope}

\end{tikzpicture}

\caption{
Game division in Case~\ref{case:triangle_common}.
}
\label{fig:triangle_common}
\end{figure}
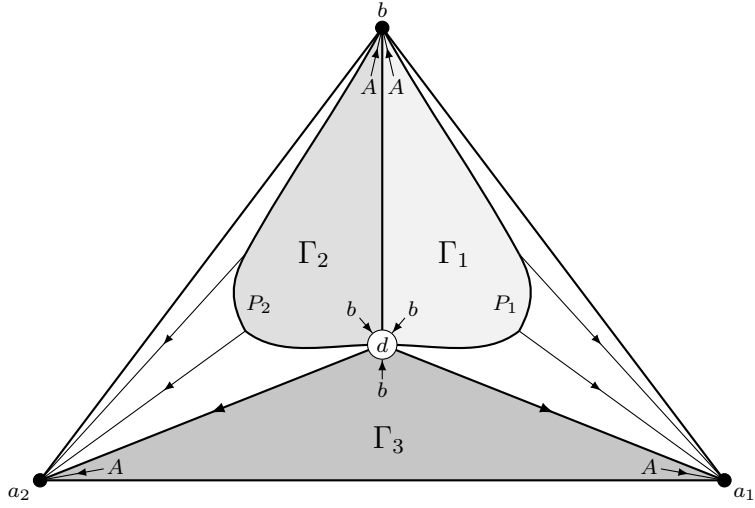

\Subcase{Vertex $d$ is adjacent to $a_1$, $a_2$, and $b$}\label{case:triangle_common}
  Figure~\ref{fig:triangle_common} depicts the game division that we use in this case.
  Let $G_1$ be the graph $\inter{a_1,d,b,a_1}$.
  Let $P_1$ be the path $NP(a_1)(d,b)$.
  Let $G_2$ be the graph $\inter{a_2,b,d,a_2}$.
  Let $P_2$ be the path $NP(a_2)(b,d)$.
  Let $G_3$ be the graph $\inter{a_1,a_2,d,a_1}$.
  Vertex $d$, each vertex in $P_1$, and each vertex in $P_2$ has three tokens.
  We divide the game into smaller games:
  \begin{itemize}
    \item $\Gamma_1=(G_1 \grminus a_1, \set{b}, d)$,
    \item $\Gamma_2=(G_2 \grminus a_2, \set{b}, d)$,
    \item $\Gamma_3=(G_3, \set{a_1,a_2}, d)$.
  \end{itemize}
  Vertex $d$ gives away a token to $a_1$, and one token to $a_2$.
  Each vertex in $P_1$ gives away a token to $a_1$.
  Each vertex in $P_2$ gives away a token to $a_2$.

  As a result, vertices $a_1$, and $a_2$ get no defect in $\Gamma_1$, $\Gamma_2$, $\Gamma_3$.
  Thus, the only vertex that can give defect to $a_1$, or $a_2$ is $b$.
  Vertex $d$ gets at most one defect in each of the games $\Gamma_1$, $\Gamma_2$, $\Gamma_3$.
  We have that $\theta(b) \subseteq \set{a_1,a_2,d}$ and each of these vertices is colored in a different round.
  Thus, vertex $b$ gets at most one defect.

\begin{figure}[h]
\begin{tikzpicture}[scale=0.9, yscale=0.75]
\tikzset{->-/.style={decoration={
  markings,
  mark=at position #1 with {\arrow{latex}}},postaction={decorate}}}

\begin{scope}[shift={(-10,0)}]

\coordinate (a2) at (-5,0) {};
\coordinate (a1) at (4.2,0) {};
\coordinate (g') at (1,4.7) {};
\coordinate (b) at (1,10) {};
\coordinate (x1) at (1.9,5) {};
\coordinate (x2) at (0.1,5) {};
\coordinate (x3) at (1,3.65) {};
\coordinate (y2) at (-1.4,5) {};
\coordinate (y3) at (1,0.6) {};
\coordinate (c) at (1,2) {};

\coordinate (x1x3) at (1.75,3.9) {};
\coordinate (x3x2) at (0.25,3.9) {};
\coordinate (x2x1_1) at (0.7,5.85) {};
\coordinate (x2x1_2) at (1.3,5.85) {};

\coordinate (x2y2) at (-1.3,4.35) {};
\coordinate (y2x2) at (-0.55,5.75) {};

%Gamma2
\filldraw[fill=gray!10, draw=none]
(y2) to[out=60,in=185] (y2x2) to[out=-5,in=90] (x2) to[out=-140,in=0] (x2y2) to[out=160, in=-90] cycle;

%Gamma2
\filldraw[fill=gray!45, draw=none]
(x3) to[out=-45,in=20] (c) to[out=-45,in=20] (y3) to[out=180,in=-150] (c) to[out=180,in=-145] cycle;

\draw[thick] (a1) -- (a2);
\draw[thick] (a1) -- (b);
\draw[thick] (a2) -- (b);

\draw[->-=.5] (x1) -- (a1);
\draw[->-=.5] (x1) -- (b);
\draw[->-=.5] (y2) to (a2);
\draw[->-=.5] (y2) to (b);
\draw[->-=.5] (y3) to (a1);
\draw[->-=.5] (y3) to (a2);
\draw[->-=.5] (x2) to (a2);
\draw[->-=.5] (x2) to (b);
\draw[->-=.5] (x1) to (b);
\draw[->-=.5] (x3) to (a1);
\draw[->-=.5] (x3) to (a2);

\draw[->-=.5] (x2x1_1) to (b);
\draw[->-=.5] (x2x1_2) to (b);
\draw[->-=.5] (x1x3) to (a1);
\draw[->-=.5] (x3x2) to (a2);

\draw[->-=.5] (y2x2) to (b);
\draw[->-=.5] (x2y2) to (a2);

\draw[->-=.5] (c) to (a2);
\draw[->-=.5] (c) to (a1);

\draw[thick] (x2) to[out=75,in=190] (x2x1_1) to[out=8,in=172] (x2x1_2) to[out=-10,in=105] (x1);
\draw[thick] (x1) to[out=-93,in=55] (x1x3) to[out=-135, in=3] (x3);
\draw[thick] (x3) to[out=-177,in=-45] (x3x2) to[out=125, in=-87] (x2);

\draw[thick] (y2) to[out=60,in=185] (y2x2) to[out=-5,in=90] (x2);
\draw[thick] (x2) to[out=-140,in=0] (x2y2) to[out=160, in=-90] (y2);

\draw[thick] (x3) to[out=-45,in=20] (c);
\draw[thick] (c) to[out=-45,in=20] (y3);
\draw[thick] (y3) to[out=180,in=-150] (c);
\draw[thick] (c) to[out=180,in=-145] (x3);

\tikzstyle{every node}=[circle,minimum size=5pt,inner sep=0pt,draw,fill]
\begin{tiny}
  \node (a2n) at (a2) {};
  \node (a1n) at (a1) {};
  \node (bn) at (b) {};
  \node (x1n) at (x1) {};
  \node (x2n) at (x2) {};
  \node (x3n) at (x3) {};
  \node (y2n) at (y2) {};
  \node (y3n) at (y3) {};
  \node (cn) at (c) {};
\end{tiny}

\begin{tiny}
\tikzstyle{every node}=[inner sep=0pt]
\node at ($(a1) + (0.3,-0.3)$) {$a_1$};
\node at ($(a2) + (-0.3,-0.3)$) {$a_2$};
\node at ($(b) + (0,0.4)$) {$b$};
\node at ($(x1) + (-0.3,-0.1)$) {$x_1$};
\node at ($(x1) + (0.35,-0.1)$) {$y_1$};
\node at ($(x2) + (+0.35,-0.1)$) {$x_2$};
\node at ($(x3) + (0,0.35)$) {$x_3$};
\node at ($(y2) + (-0.3,0.0)$) {$y_2$};
\node at ($(y3) + (0.0,-0.35)$) {$y_3$};
\node at (g') {$G'$};
\node at ($(c) + (0.4,0)$) {$c$};
\end{tiny}

%GAMMA2
\begin{tiny}
\tikzstyle{every node}=[inner sep=1pt]
\node at (-0.5,5.35) {$\Gamma_2$};
\node (G2A) at ($(y2) + (0.4,0.2)$) {$A$};
\node (G2b) at ($(x2) + (-0.7, -0.2)$) {$b$};
\end{tiny}
\draw[->-=0.8] ($(y2)+(0.6,0)$) -- (y2);
\draw[->-=0.8] ($(x2)+(-0.6,0)$) -- (x2);

%GAMMA3

\begin{tiny}
\tikzstyle{every node}=[inner sep=1pt]
\node at (0.98,2.4) {$\Gamma_3$};
\node (G2b) at ($(x3)+(0.15,-0.6)$) {$b$};
\node (G2A) at ($(y3)+(-0.2,0.6)$) {$A$};
\end{tiny}
\draw[->-=1] ($(x3)+ (0,-0.75)$) -- (x3n);
\draw[->-=1] ($(y3)+ (0,0.75)$) -- (y3n);
\end{scope}

\begin{scope}[xscale=0.8, yscale=0.9,  shift={(-3.0,3)}]

\coordinate (x1) at (3,4.5) {};
\coordinate (x2) at (-3,4.5) {};
\coordinate (x3) at (0,0) {};
\coordinate (z6) at (-3,2.5) {};
\coordinate (z6z5_1) at (-1.5,2.60) {};
\coordinate (z6z5_2) at (-0.2,2.9) {};
\coordinate (z6z5_3) at (1.1,3.4) {};

%Gamma4
\filldraw[fill=gray!10, draw=none]
(z6) to[out=92,in=-95] (x2) to[out=60,in=120] (x1) to[out=210,in=0] cycle;

%Gamma5
\filldraw[fill=gray!25, draw=none]
(x3) -- (x1) to[out=-90,in=3] cycle;

%Gamma6
\filldraw[fill=gray!45, draw=none]
(x3) to[out=-177,in=-90] (z6)--cycle;

\draw[thick] (x3) -- (x1);
\draw[thick] (x3) -- (z6);
\draw[->-=.5] (z6z5_1) to (x3);
\draw[->-=.5] (z6z5_2) to (x3);
\draw[->-=.5] (z6z5_3) to (x3);

\draw[thick] (x2) to[out=60,in=120] (x1);
\draw[thick] (x1) to[out=210,in=0] (z6);
\draw[thick] (z6) to[out=92,in=-95] (x2);

\draw[thick] (x1) to[out=-90,in=3] (x3);
\draw[thick] (x3) to[out=-177,in=-90] (z6);

\tikzstyle{every node}=[circle,minimum size=5pt,inner sep=0pt,draw,fill]
\begin{tiny}
\node (x1n) at (x1) {};
\node (x2n) at (x2) {};
\node (x3n) at (x3) {};
\node (z6n) at (z6) {};
\end{tiny}

\begin{tiny}
\tikzstyle{every node}=[inner sep=0pt]
\node at ($(x1) + (0.9,0)$) {$x_1=z_5$};
\node at ($(x2) + (-0.4,0)$) {$x_2$};
\node at ($(x3) + (0,-0.4)$) {$x_3$};
\node at ($(z6) + (-0.4,0)$) {$z_6$};
\end{tiny}

%GAME4
\tikzstyle{every node}=[inner sep=1pt]
\node at (0,4.5) {$\Gamma_4$};
\begin{tiny}
\node at (-0.1,3.3) {$P$};
\node (G4b) at ($(x2)+(0.8,0)$) {$b$};
\node (G4A) at ($(x1)+(-0.8,0)$) {$A$};
\end{tiny}
\draw[->-=1] (G4b) -- (x2n);
\draw[->-=1] (G4A) -- (x1n);

%GAME6
\tikzstyle{every node}=[inner sep=1pt]
\node at (-1.66,0.9) {$\Gamma_6$};
\begin{tiny}
\node (G6A) at ($(z6)+(0.5,-0.85)$) {$A$};
\node (G6b) at ($(x3)+(-0.9,0.4)$) {$b$};
\end{tiny}
\draw[->-=0.8] (G6b) -- (x3n);
\draw[->-=0.8] (G6A) -- (z6n);

%GAME5
\tikzstyle{every node}=[inner sep=1pt]
\node at (1.85,1.75) {$\Gamma_5$};
\begin{tiny}
\node (G5b) at ($(x3)+(0.83,0.55)$) {$b$};
\node (G5A) at ($(x1)+(-0.35,-1.0)$) {$A$};
\end{tiny}
\draw[->-=0.8] (G5b) -- (x3n);
\draw[->-=0.8] (G5A) -- (x1n);
\end{scope}

\end{tikzpicture}
\caption{
Game division in Case~\ref{case:triangle_nocommon}.
On the left: graph $G$.
Since $x_1 = y_1$, $G_1$ consists only of $x_1$.
Vertex $c$ is an $(A,b)$-cut in $\Gamma_3$.
On the right: graph $G'$ with boundary $C'$.
} 
\label{fig:triangle_nocommon}
\end{figure}
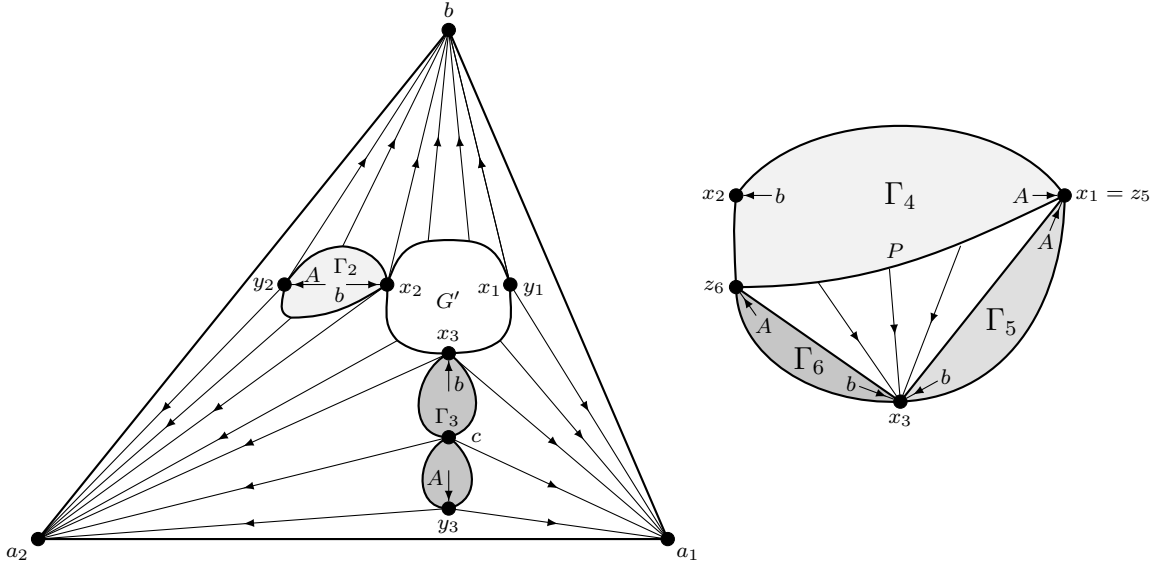

\Subcase{There is no common neighbor of $a_1$, $a_2$, and $b$}\label{case:triangle_nocommon}
  Figure~\ref{fig:triangle_nocommon} depicts the game division that we use in this case.
  Let $G_0$ be the subgraph of $G$ induced by the vertices $\set{a_1,a_2,b}$.
  Let $x_i = \maxn(a_i,b)$, and $y_i = \minn(a_i,b)$, for $i=1,2$.
  Similarly, let $x_3 = \maxn(a_1,a_2)$, and $y_3 = \minn(a_1,a_2)$.
  As there is no common neighbor of $a_1$, $a_2$, and $b$, vertices $x_1$, $x_2$, and $x_3$ are pairwise different.
  Let $G_i$, for $i=1,2,3$, be the connected component of $G \grminus \set{a_1,a_2,b,x_1,x_2,x_3}$ that contains vertex $y_i$.
  In particular, if $x_i = y_i$ then $G_i$ is an empty graph.
  Let $G'$ be the graph obtained from $G$ by removing $a_1$, $a_2$, $b$, $G_1$, $G_2$, and $G_3$.
  Observe that the choice of $x_1$, $x_2$, $x_3$ guarantees that the boundary walk $C'$ of $G'$ is a simple cycle, and that each vertex in $C'$ except $x_1$, $x_2$, $x_3$ is a neighbor of exactly one of the vertices $a_1$, $a_2$, or $b$.
  Let $z_5$ be the first (closest to $x_1$) neighbor of $x_3$ on the path $C'[x_1,x_3]$.
  Such a neighbor exists, and it is possible that $z_5 = x_1$.
  Similarly, let $z_6$ be the last (closest to $x_2$) neighbor of $x_3$ on the path $C'[x_3,x_2]$.
  It is possible that $z_6 = x_2$.
  Let $P$ denote the inverted path $NC(x_3)(z_6,z_5)$.
  Let $G_4$ be the graph $\inter{z_5,P,C'[z_6,z_5]}$.
  Let $G_5$ be the graph $\inter{C'[z_5,x_3],z_5}$.
  Let $G_6$ be the graph $\inter{C'[x_3,z_6],x_3}$.
  Update each $G_i$, for $i=1,2,3$, by adding vertex $x_i$ to it.
  We divide the game into smaller games:
  \begin{itemize}
    \item $\Gamma_0=(G_0,A,b)$,
    \item $\Gamma_1=(G_1,\set{y_1},x_1)$ (if $x_1 = y_1$, $\Gamma_1$ is not used),
    \item $\Gamma_2=(G_2,\set{y_2},x_2)$ (if $x_2 = y_2$, $\Gamma_2$ is not used),
    \item $\Gamma_3=(G_3,\set{y_3},x_3)$ (if $x_3 = y_3$, $\Gamma_3$ is not used),
    \item $\Gamma_4=(G_4,\set{x_1},x_2)$,
    \item $\Gamma_5=(G_5,\set{z_5},x_3)$,
    \item $\Gamma_6=(G_6,\set{z_6},x_3)$.
  \end{itemize}

  Each vertex adjacent to $a_1$, $a_2$, $b$ gives away a token to each of the adjacent vertices $a_1$, $a_2$, $b$.
  This way we get that each vertex $a_1$, $a_2$, $b$ receives at most one defect and that $\theta(a_1)$, and $\theta(a_2)$ are contained in $\set{b}$.

  There is no common neighbor of $a_1$, $a_2$, and $b$, so each vertex gives away at most two tokens.
  Vertices that give away exactly two tokens to special vertices are $x_i$, $y_i$, and $(\set{y_i},x_i)$-cuts in $G_i$, for $i=1,2,3$.
  Thus, each vertex in $G_1$, $G_2$, $G_3$ has enough tokens for the games $\Gamma_1$, $\Gamma_2$, $\Gamma_3$.
  Each vertex $x_1$, $x_2$, $x_3$ gets at most one defect in the games $\Gamma_1$, $\Gamma_2$, $\Gamma_3$.
  Each vertex $x_1$, $x_2$, $x_3$ gets at most two defects in the games $\Gamma_4$, $\Gamma_5$, $\Gamma_6$.

  Vertices in $C'$ other than $x_1$, $x_2$, $x_3$ give away only one token to special vertices and have two tokens of value $3$ left.
  When vertex $z_5$ is different than $x_1$, then it is marked with a token of value $2$ in game $\Gamma_4$ when $x_3$ is colored in this round.
  This way we get that vertex $z_5$ different than $x_1$ gets at most two defects in $\Gamma_4$ and one defect in $\Gamma_5$ if it gets colored in the same round as $x_3$.
  If $z_5$ is colored in a different round, then it gets at most three defects in $\Gamma_4$ and no defect in $\Gamma_5$.

  Similarly, when vertex $z_6$ is different than $x_2$, then it is marked with a token of value $2$ in game $\Gamma_4$ when $x_3$ is colored in this round.

\Case{Special vertex $b$ is adjacent to an element of $A$}\label{case:ab}
  Without loss of generality, assume that $b$ is adjacent to $a_1$.
  We divide this case depending whether $\set{a_1,b}$ is an edge of $C$ or not.

\begin{figure}[h]
\centering
\begin{tikzpicture}[scale=0.9, yscale=0.75]

\tikzset{->-/.style={decoration={
  markings,
  mark=at position #1 with {\arrow{latex}}},postaction={decorate}}}

\begin{scope}
\coordinate (a1) at (-0.75,5) {};
\coordinate (a2) at (0.75,5) {};
\coordinate (b) at (0,0) {};
\coordinate (ba1) at (-2.5,2.5) {};
\coordinate (a2b) at (2.5,2.5) {};

%Gamma1
\filldraw[fill=gray!10, draw=none]
  (a1) to[out=0,in=180] (a2) to[out=-15,in=90] (a2b) to[out=-90,in=0] (b) -- cycle;

%Gamma2
\filldraw[fill=gray!45, draw=none]
  (a1) -- (b) to[out=180,in=-90] (ba1) to[out=90,in=-175] cycle;

\tikzstyle{every node}=[circle,minimum size=5pt,inner sep=0pt,draw,fill]
\begin{tiny}
  \node (a1n) at (a1) {};
  \node (a2n) at (a2) {};
  \node (bn) at (b) {};
\end{tiny}

\begin{tiny}
\tikzstyle{every node}=[inner sep=0pt]
\node at  ($(a1) + (0.0,0.35)$) {$a_1$};
\node at  ($(a2) + (0.1,0.35)$) {$a_2$};
\node at  ($(b) + (-0,-0.4)$) {$b$};
\end{tiny}

\draw[thick] (b) -- (a1);
\draw[thick] (b) to[out=180,in=-90] (ba1);
\draw[thick] (ba1) to[out=90,in=-175] (a1);
\draw[thick] (a1) to[out=0,in=180] (a2);
\draw[thick] (a2) to[out=-15,in=90] (a2b);
\draw[thick] (a2b) to[out=-90,in=0] (b);

%GAME2
\tikzstyle{every node}=[inner sep=1pt]
\node at (-1.4,2.4) {$\Gamma_2$};
\begin{tiny}
\node (G2b) at ($(a1)+(-0.35,-0.7)$) {$b$};
\node (G2A) at ($(b)+(-0.5,0.65)$) {$A$};
\end{tiny}

\draw[->-=1] (G2A) -- (bn);
\draw[->-=1] (G2b) -- (a1n);

%GAME1
\tikzstyle{every node}=[inner sep=1pt]
\node at (0.9,2.4) {$\Gamma_1$};
\begin{tiny}
\node (G1A1) at ($(a1)+(0.45,-0.65)$) {$A$};
\node (G1A2) at ($(a2) + (-0.4,-0.65)$) {$A$};
\node (G1b) at ($(b) + (0.4, 0.65)$) {$b$};
\end{tiny}
\draw[->-=1] (G1A1) -- (a1n);
\draw[->-=1] (G1A2) -- (a2n);
\draw[->-=1] (G1b) -- (bn);

\end{scope}
\end{tikzpicture}

\caption{
Game division in Case~\ref{case:ab_chord}.
}
\label{fig:ab_chord}
\end{figure}
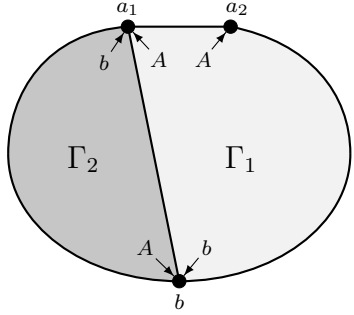

\Subcase{$\set{a_1,b}$ is a chord of $C$}\label{case:ab_chord}
  Figure~\ref{fig:ab_chord} depicts the game division that we use in this case.
  Observe that in this case $a_1$ is not a $C$-neighbor of $b$ and $a_1$ can get defect from any vertex in $G$.
  Let $G_1$ be $\inter{C[a_1,b],a_1}$.
  Let $G_2$ be $\inter{C[b,a_1],b}$.
  We divide the game into smaller games:
  \begin{itemize}
    \item $\Gamma_1=(G_1, A, b)$,
    \item $\Gamma_2=(G_2, \set{b}, a_1)$.
  \end{itemize}
  When vertex $b$ gets defect in $\Gamma_2$ then $b$ gets the defect from $a_1$.
  Similarly, when vertex $a_1$ gets defect in $\Gamma_1$ then $a_1$ gets the defect from $b$.
  Thus, each vertex $a_1$, $b$ gets at most one defect in both games $\Gamma_1$, $\Gamma_2$.

\Subcase{$\set{a_1,b}$ is an edge of $C$}\label{case:ab_cycle}
  Observe that $\set{a_2,b}$ is not an edge of $C$, as then $C$ would be a triangle and we could apply Case~\ref{case:triangle}.
  We divide this case further, depending whether $a_1$ has a neighbor in $C$ other than $a_2$ and $b$.
  Figure~\ref{fig:ab_cycle} depicts the game divisions that we use in the subcases.

\begin{figure}[h]
\centering
\begin{tikzpicture}[scale=0.9, yscale=0.75]

\tikzset{->-/.style={decoration={
  markings,
  mark=at position #1 with {\arrow{latex}}},postaction={decorate}}}

\begin{scope}[shift={(-7.5,0)}]
  \coordinate (b) at (-1,5) {};
  \coordinate (a1) at (1,5) {};
  \coordinate (a2) at (2.5,3.5) {};
  \coordinate (c) at (-1,0) {};
  \coordinate (b-) at (-2.5,3.5) {};
  \coordinate (e) at (0,-0.1) {};
  \coordinate (d) at (0,2.4) {};
  \coordinate (b-c) at (-1.7,2.47) {};
  \coordinate (cd) at (-0.45,1.88) {};
  \coordinate (d1a2) at (0.85,3.1) {};
  \coordinate (d2a2) at (1.7,3.4) {};

%Gamma1
\filldraw[fill=gray!10, draw=none]
  (c) to[out=160,in=-100] (b-) to[out=-45,in=95] cycle;

\filldraw[fill=gray!10, draw=none]
  (a2) to[out=-80,in=1] (e) to[out=179,in=-13] (c) to[out=85,in=-135] (d) to[out=45,in=180] cycle;

%Gamma2
\filldraw[fill=gray!45, draw=none]
  (d) -- (b) -- (a1) -- cycle;

\begin{tiny}
\tikzstyle{every node}=[circle,minimum size=5pt,inner sep=0pt,draw,fill]
  \node (a1n) at (a1) {};
  \node (a2n) at (a2) {};
  \node (dn) at (d) {};
  \node (cn) at (c) {};
  \node (b-n) at (b-) {};
  \node (bn) at (b) {};
\end{tiny}

\draw[thick,->-=0.5] (b-) to (b);
\draw[thick] (b) to (a1);
\draw[thick] (a1) to (a2);
\draw[thick] (a2) to[out=-80,in=1] (e);
\draw[thick] (e) to[out=179,in=-13] (c);
\draw[thick] (c) to[out=160,in=-100] (b-);

\draw[thick] (b-) to[out=-45,in=95] (c);
\draw[thick] (c) to[out=85,in=-135] (d);
\draw[thick] (d) to[out=45,in=180] (a2);

\draw[thick] (d) -- (b);
\draw[->-=0.55, thick] (d) -- (a1);

\draw[->-=0.5] (c) -- (b);
\draw[->-=0.5] (b-c) -- (b);
\draw[->-=0.5] (cd) -- (b);
\draw[->-=0.5] (d1a2) -- (a1);
\draw[->-=0.5] (d2a2) -- (a1);

\begin{tiny}
\tikzstyle{every node}=[inner sep=0pt]
\node at  ($(b) + (0,0.35)$) {$b$};
\node at  ($(a1) + (0.1,0.35)$) {$a_1$};
\node at  ($(a2) + (0.35,0.0)$) {$a_2$};
\node at  ($(c) + (0.0,-0.4)$) {$c$};
\node at  ($(b-) + (-0.35,0.0)$) {$b^-$};
\node at  ($(d) + (0.0,-0.4)$) {$d$};
\end{tiny}

%Gamma1
\tikzstyle{every node}=[inner sep=1pt]
\node at (1,1.5) {$\Gamma_1$};
\begin{tiny}
\node at (1.3,3.05) {$P$};
\node at (-0.52,1.3) {$Q$};
\end{tiny}
\begin{tiny}
\node (G1A) at ($(a2) + (-0.55,-0.55)$) {$A$};
\node (G1b) at ($(b-) + (0.18,-0.8)$) {$b$};
\end{tiny}
\draw[->-=1] (G1A) -- (a2n);
\draw[->-=1] (G1b) -- (b-n);

%Gamma2
\tikzstyle{every node}=[inner sep=1pt]
\node at (0,4) {$\Gamma_2$};
\begin{tiny}
\node at (-1.56,1.7) {$Q$};
\node (G2A1) at ($(a1) + (-0.6,-0.35)$) {$A$};
\node (G2A2) at ($(d) + (0,0.8)$) {$A$};
\node (G2b) at ($(b) + (0.55,-0.35)$) {$b$};
\end{tiny}
\draw[->-=1] (G2A1) -- (a1n);
\draw[->-=1] (G2A2) -- (dn);
\draw[->-=1] (G2b) -- (bn);

\end{scope}

\begin{scope}[shift={(0,0)}]
  \coordinate (b) at (-1,5) {};
  \coordinate (a1) at (1,5) {};
  \coordinate (a2) at (2.5,3.5) {};
  \coordinate (d) at (-1,0) {};
  \coordinate (b-) at (-2.5,3.5) {};
  \coordinate (e) at (0,-0.1) {};

\filldraw[fill=gray!45, draw=none]
  (b) -- (a1) -- (d) to[out=160,in=-100] (b-) to[out=77,in=185] cycle;

\filldraw[fill=gray!10, draw=none]
  (a1) -- (a2) to[out=-80,in=1] (e) to[out=179,in=-13] (d) -- cycle;

\begin{tiny}
\tikzstyle{every node}=[circle,minimum size=5pt,inner sep=0pt,draw,fill]
  \node (a1n) at (a1) {};
  \node (a2n) at (a2) {};
  \node (dn) at (d) {};
  \node (bn) at (b) {};
\end{tiny}

\draw[thick] (b-) to[out=77,in=185] (b);
\draw[thick] (b) to (a1);
\draw[thick] (a1) to (a2);
\draw[thick] (a2) to[out=-80,in=1] (e);
\draw[thick] (e) to[out=179,in=-13] (d);
\draw[thick] (d) to[out=160,in=-100] (b-);

\draw[->-=.5,thick] (d) -- (a1);

\begin{tiny}
\tikzstyle{every node}=[inner sep=0pt]
\node at  ($(b) + (0.0,0.4)$) {$b$};
\node at  ($(a1) + (0.1,0.35)$) {$a_1$};
\node at  ($(a2) + (0.35,0.0)$) {$a_2$};
\node at  ($(d) + (0,-0.4)$) {$d$};
\end{tiny}

%GAME2
\tikzstyle{every node}=[inner sep=1pt]
\node at (-1.1,2.8) {$\Gamma_2$};
\begin{tiny}
\node (G2b) at ($(b) + (0.3,-0.7)$) {$b$};
\node (G2A1) at ($(a1) + (-0.6,-0.5)$) {$A$};
\node (G2A2) at ($(d) + (-0.3,0.75)$) {$A$};
\end{tiny}

\draw[->-=1] (G2A1) -- (a1n);
\draw[->-=1] (G2A2) -- (dn);
\draw[->-=1] (G2b) -- (bn);

%GAME1
\tikzstyle{every node}=[inner sep=1pt]
\node at (1.1,2.2) {$\Gamma_1$};
\begin{tiny}
\node (G1A1) at ($(a1) + (0.15,-0.85)$) {$A$};
\node (G1A2) at ($(d) + (0.65,0.4)$) {$A$};
\node (G1b) at ($(a2) + (-0.55,-0.3)$) {$b$};
\end{tiny}
\draw[->-=1] (G1A1) -- (a1n);
\draw[->-=1] (G1A2) -- (dn);
\draw[->-=1] (G1b) -- (a2n);

\end{scope}

\end{tikzpicture}

\caption{
Game division in Cases~\ref{case:ab_cycle_nochord} and~\ref{case:ab_cycle_chord}, respectively.
On the left: vertex $c$ is an $(A,b)$-cut in $\Gamma_1$.
}
\label{fig:ab_cycle}
\end{figure}
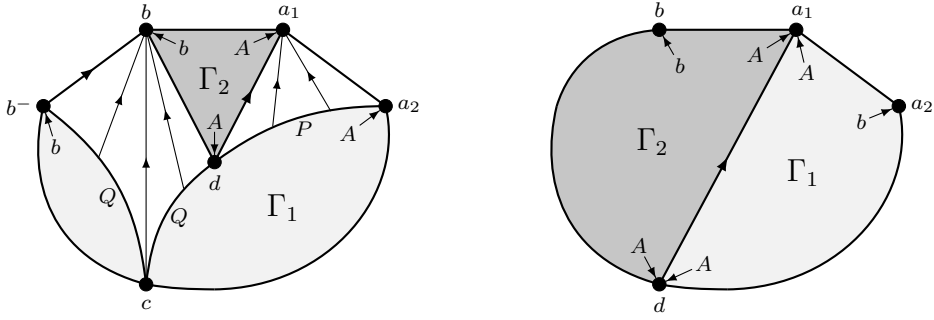

\Subsubcase{$a_1$ is not incident to a chord of $C$}\label{case:ab_cycle_nochord}
  Let $d$ be the vertex $\maxn(a_1,b)$.
  Assume that $d$ is different than $a_2$, as then the edge $\set{a_2,b}$ would be a chord of $C$ and we could apply Case~\ref{case:ab_chord}.
  Observe that $d$ is an internal vertex, as $a_1$ is not incident to a chord of $C$.

  Let $P$ be the inverted path $NP(a_1)(a_2,d]$.
  Let $Q$ be the inverted path $NP(b)(d,b^-]$.
  Let $G_1$ be $\inter{Q,P,C[a_2,b^-]}$.
  Let $G_2$ be $\inter{b,a_1,d,b}$.
  We divide the game into smaller games:
  \begin{itemize}
    \item $\Gamma_1=(G_1, \set{a_2},b^-)$,
    \item $\Gamma_2=(G_2, \set{a_1,d},b)$.
  \end{itemize}
  Vertices in $P$ give away a token to $a_1$.
  Vertices in $Q$ give away a token to $b$.
  Observe that vertices that are both in $C$ and $Q$ are $(\set{a_2},b^-)$-cuts in $G_1$.
  Vertex $d$ is marked with the token of value $2$ in $\Gamma_1$ if $b$ is marked in the same round.
  If it is the case, then $d$ gets at most one defect in $\Gamma_2$ (from $b$) and at most two defects in $\Gamma_1$.
  Otherwise, should vertex $d$ be colored in this round, it gets no defect in $\Gamma_2$.

\Subsubcase{$\set{a_1,d}$ is a chord of $C$}\label{case:ab_cycle_chord}
  Let $G_1$ be $\inter{C[a_1,d],a_1}$.
  Let $G_2$ be $\inter{C[d,a_1],d}$.
  We divide the game into smaller games:
  \begin{itemize}
    \item $\Gamma_1=(G_1, \set{a_1,d}, a_2)$,
    \item $\Gamma_2=(G_2, \set{a_1,d}, b)$.
  \end{itemize}
  Vertex $d$ gives away a token to $a_1$.
  Vertex $a_1$ gets no defect in $\Gamma_1$.
  Vertex $d$ gets at most one defect in $\Gamma_1$ and at most one defect in $\Gamma_2$.

\Case{A chord $\set{a_1,d}$ of $C$ separates $a_2$ from $b$}\label{case:chord}
  Observe, that the same game division as in Case~\ref{case:ab_cycle_chord} (see Figure~\ref{fig:ab_cycle}) works also in this case.

\begin{figure}[h]
\begin{tikzpicture}[scale=1, yscale=0.85]

\tikzset{->-/.style={decoration={
  markings,
  mark=at position #1 with {\arrow{latex}}},postaction={decorate}}}

\coordinate (a1) at (0,0) {};
\coordinate (a2) at (0.3,2) {};
\coordinate (p2) at (2.5,3.3) {};
\coordinate (p3) at (4.5,3.65) {};

\coordinate (up) at (5.5,3.7) {};

\coordinate (p_{m-2}) at (6.5,3.65) {};
\coordinate (p_{m-1}) at (8.5,3.3) {};
\coordinate (b) at (10.7,2) {};
\coordinate (b+) at (11,0) {};

\coordinate (q01) at (0.75,0) {};
\coordinate (q1) at (1.5,0) {};
\coordinate (q11) at (2.2,0) {};
\coordinate (q12) at (2.8,0) {};
\coordinate (q2) at (3.5,0) {};
\coordinate (q21) at (4.4,0) {};
\coordinate (q3) at (5.3,0) {};

\coordinate (q_{m-3}) at (6,0) {};
\coordinate (q_{m-3,1}) at (7,0) {};
\coordinate (q_{m-2}) at (8,0) {};
\coordinate (q_{m-2,1}) at (9,0) {};
\coordinate (q_{m-2,2}) at (10,0) {};

\coordinate (low) at (5.5,-2.0) {};

\filldraw[fill=gray!10, draw=none]
  (q1) -- (a2) to[out=50,in=200] (p2) -- cycle;

\filldraw[fill=gray!20, draw=none]
  (q2) -- (p2) to[out=15,in=187] (p3) -- cycle;

\filldraw[fill=gray!10, draw=none]
  (q_{m-2}) -- (p_{m-2}) to[out=-3,in=167] (p_{m-1}) -- cycle;

\filldraw[fill=gray!20, draw=none]
  (b+) -- (p_{m-1}) to[out=-20,in=130] (b) -- cycle;

\filldraw[fill=gray!35, draw=none]
  (b+) to[out=240, in=0] (low) to[out=180, in=-60] (a1) -- cycle;

\draw[thick] (a1) -- (a2);
\draw[thick, red] (a2) to[out=50,in=200] (p2);
\draw[thick, red] (p2) to[out=15,in=187] (p3);
\draw[dashed, red] (p3) to[out=2,in=180] (up);

\draw[dashed, red] (up) to[out=0,in=178] (p_{m-2});
\draw[thick,red] (p_{m-2}) to[out=-3,in=165] (p_{m-1});
\draw[thick,red] (p_{m-1}) to[out=-20,in=130] (b);
\draw[thick] (b+) -- (b);

\draw[thick, blue] (a1) -- (q3);
\draw[dashed,blue] (q3) -- (q_{m-3});
\draw[thick, blue] (q_{m-3}) -- (b+);

\draw[thick] (b+) to[out=240, in=0] (low);
\draw[thick] (low) to[out=180, in=-60] (a1);

\draw[->-=.5, thick] (q1) to (a2);
\draw[->-=.5] (q01) to (a2);

\draw[thick] (q1) to (p2);
\draw[->-=.5] (q11) to (p2);
\draw[->-=.5] (q12) to (p2);
\draw[->-=.5,thick] (q2) to (p2);

\draw[thick] (q2) to (p3);
\draw[->-=.5] (q21) to (p3);
\draw[->-=.5, thick] (q3) to (p3);
%\draw[color=gray] (q31) to (p4);
%\draw[thick] (q4) to (p4);

\draw[thick] (q_{m-3}) to (p_{m-2});
\draw[->-=.5] (q_{m-3,1}) to (p_{m-2});
\draw[->-=.5,thick] (q_{m-2}) to (p_{m-2});
\draw[->-=.5] (q_{m-2,1}) to (p_{m-1});
\draw[->-=.5] (q_{m-2,2}) to (p_{m-1});
\draw[thick] (q_{m-2}) to (p_{m-1});
\draw[->-=.5, thick] (b+) to (p_{m-1});

\begin{tiny}
\tikzstyle{every node}=[inner sep=0pt]
\node at ($(a1) + (-0.65,0)$) {$a_1 = q_0$};
\node at ($(a2) + (-0.65,0)$) {$a_2 = p_1$};
\node at ($(p2) + (0,0.3)$) {$p_2$};
\node at ($(p3) + (0,0.3)$) {$p_3$};
\node at ($(p_{m-2}) + (0,0.35)$) {$p_{m-2}$};
\node at ($(p_{m-1}) + (0,0.35)$) {$p_{m-1} = p_l$};
\node at ($(b) + (0.65,0)$) {$b = p_m$};
\node at ($(b+) + (0.8,0)$) {$\begin{array}{l}b^+ = q_m \\ = q_{m-1} \\ = c\end{array}$};
\node at ($(q1) + (0,-0.35)$) {$q_1$};
\node at ($(q2) + (0,-0.35)$) {$q_2$};
\node at ($(q3) + (-0.1,-0.35)$) {$q_3$};
\node at ($(q_{m-3}) + (0.15,-0.35)$) {$q_{m-3}$};
\node at ($(q_{m-2}) + (0,-0.35)$) {$q_{m-2}$};

\tikzstyle{every node}=[circle,minimum size=5pt,inner sep=0pt,draw,fill]
\node (a1n) at (a1) {};
\node (a2n) at (a2) {};
\node (p2n) at (p2) {};
\node (p3n) at (p3) {};

\node (p_{m-2}n) at (p_{m-2}) {};
\node (p_{m-1}n) at (p_{m-1}) {};
\node (bn) at (b) {};
\node (b+n) at (b+) {};

\node (q1n) at (q1) {};
\node (q2n) at (q2) {};
\node (q3n) at (q3) {};

\node (q_{m-3}n) at (q_{m-3})  {};
\node (q_{m-2}n) at (q_{m-2}) {};

\end{tiny}

%GAMMA1
\tikzstyle{every node}=[inner sep=1pt]
\node at (5.65,-1.2) {$\Gamma_1$};
\begin{tiny}
\node (G2A) at ($(a1n) + (0.55,-0.3)$) {$A$};
\node (G2b) at ($(b+n) + (-0.6, -0.3)$) {$b$};
\end{tiny}
\draw[->-=0.8] (G2A) -- (a1n);
\draw[->-=0.8] (G2b) -- (b+n);

%GAMMA2
\tikzstyle{every node}=[inner sep=1pt]
\node at (1.5,1.8) {$\Gamma_2$};
\begin{tiny}
\node (G3A1) at ($(a2n) + (0.65,-0.1)$) {$A$};
\node (G3A2) at ($(q1n) + (-0.05,0.7)$) {$A$};
\node (G3b) at ($(p2n) + (-0.45, -0.55)$) {$b$};
\end{tiny}
\draw[->-=1] (G3A1) -- (a2n);
\draw[->-=0.8] (G3A2) -- (q1n);
\draw[->-=1] (G3b) -- (p2n);

%GAMMA4
\tikzstyle{every node}=[inner sep=1pt]
\node at (3.5,2.1) {$\Gamma_3$};
\begin{tiny}
\node (G4A1) at ($(p2n) + (0.55,-0.3)$) {$A$};
\node (G4A2) at ($(q2n) + (0.0,0.8)$) {$A$};
\node (G4b) at ($(p3n) + (-0.45, -0.45)$) {$b$};
\end{tiny}
\draw[->-=1] (G4A1) -- (p2n);
\draw[->-=0.8] (G4A2) -- (q2n);
\draw[->-=1] (G4b) -- (p3n);

%GAMMAm
\tikzstyle{every node}=[inner sep=1pt]
\node at (7.75,2.15) {$\Gamma_{m-1}$};
\begin{tiny}
\node (GmA1) at ($(p_{m-2}n) + (0.55,-0.5)$) {$A$};
\node (GmA2) at ($(q_{m-2}n) + (-0.1,0.8)$) {$A$};
\node (Gmb) at ($(p_{m-1}n) + (-0.45, -0.45)$) {$b$};
\end{tiny}
\draw[->-=1] (GmA1) -- (p_{m-2}n);
\draw[->-=0.8] (GmA2) -- (q_{m-2}n);
\draw[->-=1] (Gmb) -- (p_{m-1}n);

%GAMMAm1
\tikzstyle{every node}=[inner sep=1pt]
\node at (10.05,2.1) {$\Gamma_{m}$};
\begin{tiny}
\node (G_{m+1}A1) at ($(p_{m-1}n) + (0.65,-0.52)$) {$A$};
\node (G_{m+1}A2) at ($(b+n) + (-0.35,0.85)$) {$A$};
\node (G_{m+1}b) at ($(b) + (-0.08, -0.7)$) {$b$};
\end{tiny}
\draw[->-=0.8] (G_{m+1}A1) -- (p_{m-1}n);
\draw[->-=0.8] (G_{m+1}A2) -- (b+n);
\draw[->-=1] (G_{m+1}b) -- (bn);

\end{tikzpicture}
\caption{
Game division in Case~\ref{case:final}.
In this figure $q_{m-1} = q_{m}$, $l=m-1$, and $c=b^{+}$.
Path $P$ is depicted in red, path $Q$ is depicted in blue.
}
\label{fig:final_2}
\end{figure}
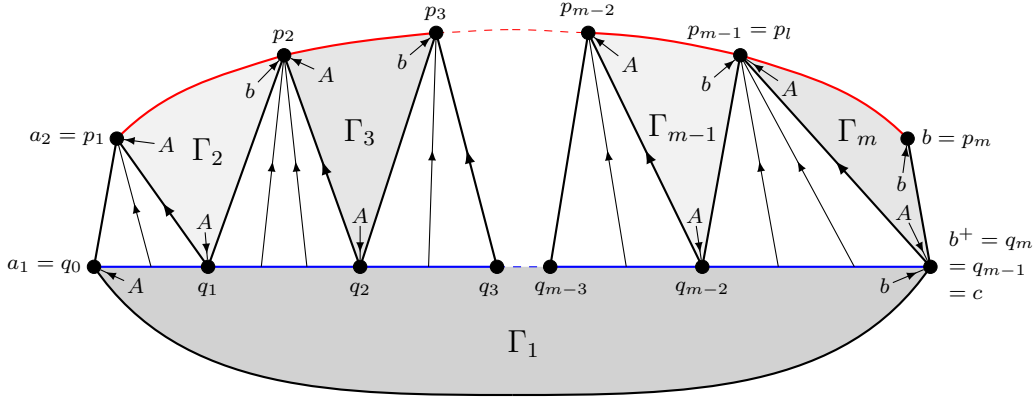
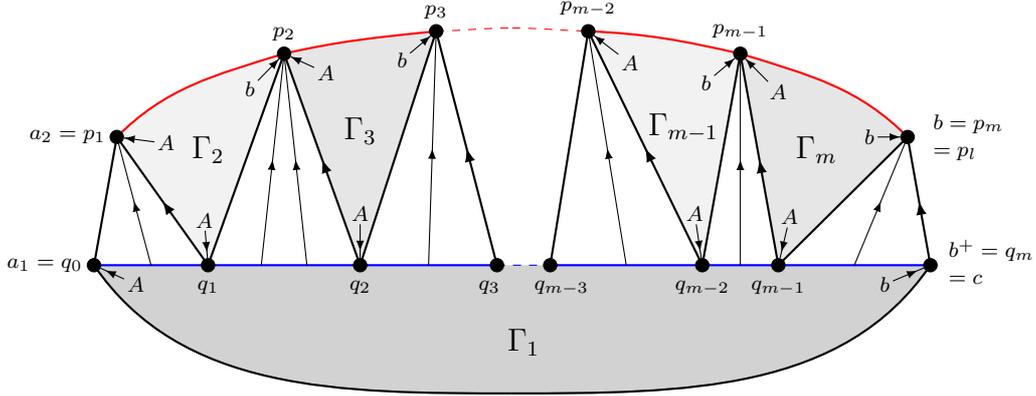
\begin{figure}[h]
\begin{tikzpicture}[scale=1, yscale=0.85]

\tikzset{->-/.style={decoration={
  markings,
  mark=at position #1 with {\arrow{latex}}},postaction={decorate}}}

\coordinate (a1) at (0,0) {};
\coordinate (a2) at (0.3,2) {};
\coordinate (p2) at (2.5,3.3) {};
\coordinate (p3) at (4.5,3.65) {};

\coordinate (up) at (5.5,3.7) {};

\coordinate (p_{m-2}) at (6.5,3.65) {};
\coordinate (p_{m-1}) at (8.5,3.3) {};
\coordinate (b) at (10.7,2) {};
\coordinate (b+) at (11,0) {};

\coordinate (q01) at (0.75,0) {};
\coordinate (q1) at (1.5,0) {};
\coordinate (q11) at (2.2,0) {};
\coordinate (q12) at (2.8,0) {};
\coordinate (q2) at (3.5,0) {};
\coordinate (q21) at (4.4,0) {};
\coordinate (q3) at (5.3,0) {};

\coordinate (q_{m-3}) at (6,0) {};
\coordinate (q_{m-3,1}) at (7,0) {};
\coordinate (q_{m-2}) at (8,0) {};
\coordinate (q_{m-2,1}) at (8.5,0) {};
\coordinate (q_{m-1}) at (9,0) {};
\coordinate (q_{m-1,1}) at (10,0) {};

\coordinate (low) at (5.5,-2.0) {};

\filldraw[fill=gray!10, draw=none]
  (q1) -- (a2) to[out=50,in=200] (p2) -- cycle;

\filldraw[fill=gray!20, draw=none]
  (q2) -- (p2) to[out=15,in=187] (p3) -- cycle;

\filldraw[fill=gray!10, draw=none]
  (q_{m-2}) -- (p_{m-2}) to[out=-3,in=167] (p_{m-1}) -- cycle;

\filldraw[fill=gray!20, draw=none]
  (q_{m-1}) -- (p_{m-1}) to[out=-20,in=130] (b) -- cycle;

\filldraw[fill=gray!35, draw=none]
  (b+) to[out=240, in=0] (low) to[out=180, in=-60] (a1) -- cycle;

\draw[thick] (a1) -- (a2);
\draw[thick, red] (a2) to[out=50,in=200] (p2);
\draw[thick, red] (p2) to[out=15,in=187] (p3);
\draw[dashed, red] (p3) to[out=2,in=180] (up);

\draw[dashed, red] (up) to[out=0,in=178] (p_{m-2});
\draw[thick,red] (p_{m-2}) to[out=-3,in=165] (p_{m-1});
\draw[thick, red] (p_{m-1}) to[out=-20,in=130] (b);
\draw[->-=.5, thick] (b+) -- (b);

\draw[thick, blue] (a1) -- (q3);
\draw[dashed, blue] (q3) -- (q_{m-3});
\draw[thick, blue] (q_{m-3}) -- (b+);

\draw[thick] (b+) to[out=240, in=0] (low);
\draw[thick] (low) to[out=180, in=-60] (a1);

\draw[->-=.5, thick] (q1) to (a2);
\draw[->-=.5] (q01) to (a2);

\draw[thick] (q1) to (p2);
\draw[->-=.5] (q11) to (p2);
\draw[->-=.5] (q12) to (p2);
\draw[->-=.5,thick] (q2) to (p2);

\draw[thick] (q2) to (p3);
\draw[->-=.5] (q21) to (p3);
\draw[->-=.5, thick] (q3) to (p3);

\draw[thick] (q_{m-3}) to (p_{m-2});
\draw[->-=.5] (q_{m-3,1}) to (p_{m-2});
\draw[->-=.5,thick] (q_{m-2}) to (p_{m-2});
\draw[thick] (q_{m-2}) to (p_{m-1});
\draw[->-=.5] (q_{m-2,1}) to (p_{m-1});
\draw[->-=.5,thick] (q_{m-1}) to (p_{m-1});
\draw[thick] (q_{m-1}) to (b);
\draw[->-=.5] (q_{m-1,1}) to (b);

\begin{tiny}
\tikzstyle{every node}=[inner sep=0pt]
\node at  ($(a1) + (-0.65,0)$) {$a_1 = q_0$};
\node at ($(a2) + (-0.65,0)$) {$a_2 = p_1$};
\node at ($(p2) + (0,0.3)$) {$p_2$};
\node at ($(p3) + (0,0.3)$) {$p_3$};
\node at ($(p_{m-2}) + (0,0.35)$) {$p_{m-2}$};
\node at ($(p_{m-1}) + (0,0.35)$) {$p_{m-1}$};
\node at ($(b) + (0.8,0)$) {$\begin{array}{l}b = p_m \\ = p_l\end{array}$};
\node at ($(b+) + (0.8,0)$) {$\begin{array}{l}b^+ = q_m \\ = c\end{array}$};
\node at ($(q1) + (0,-0.35)$) {$q_1$};
\node at ($(q2) + (0,-0.35)$) {$q_2$};
\node at ($(q3) + (-0.1,-0.35)$) {$q_3$};
\node at ($(q_{m-3}) + (0.15,-0.35)$) {$q_{m-3}$};
\node at ($(q_{m-2}) + (0,-0.35)$) {$q_{m-2}$};
\node at ($(q_{m-1}) + (0,-0.35)$) {$q_{m-1}$};

\tikzstyle{every node}=[circle,minimum size=5pt,inner sep=0pt,draw,fill]
\node (a1n) at (a1) {};
\node (a2n) at (a2) {};
\node (p2n) at (p2) {};
\node (p3n) at (p3) {};

\node (p_{m-2}n) at (p_{m-2}) {};
\node (p_{m-1}n) at (p_{m-1}) {};
\node (bn) at (b) {};
\node (b+n) at (b+) {};

\node (q1n) at (q1) {};
\node (q2n) at (q2) {};
\node (q3n) at (q3) {};

\node (q_{m-3}n) at (q_{m-3})  {};
\node (q_{m-2}n) at (q_{m-2}) {};
\node (q_{m-1}n) at (q_{m-1}) {};

\end{tiny}

%GAMMA2
\tikzstyle{every node}=[inner sep=1pt]
\node at (5.65,-1.2) {$\Gamma_1$};
\begin{tiny}
\node (G2A) at ($(a1n) + (0.55,-0.3)$) {$A$};
\node (G2b) at ($(b+n) + (-0.6, -0.3)$) {$b$};
\end{tiny}
\draw[->-=0.8] (G2A) -- (a1n);
\draw[->-=0.8] (G2b) -- (b+n);

%GAMMA3
\tikzstyle{every node}=[inner sep=1pt]
\node at (1.5,1.8) {$\Gamma_2$};
\begin{tiny}
\node (G3A1) at ($(a2n) + (0.65,-0.1)$) {$A$};
\node (G3A2) at ($(q1n) + (-0.05,0.7)$) {$A$};
\node (G3b) at ($(p2n) + (-0.45, -0.55)$) {$b$};
\end{tiny}
\draw[->-=1] (G3A1) -- (a2n);
\draw[->-=0.8] (G3A2) -- (q1n);
\draw[->-=1] (G3b) -- (p2n);

%GAMMA4
\tikzstyle{every node}=[inner sep=1pt]
\node at (3.5,2.1) {$\Gamma_3$};
\begin{tiny}
\node (G4A1) at ($(p2n) + (0.55,-0.3)$) {$A$};
\node (G4A2) at ($(q2n) + (0.0,0.8)$) {$A$};
\node (G4b) at ($(p3n) + (-0.45, -0.45)$) {$b$};
\end{tiny}
\draw[->-=1] (G4A1) -- (p2n);
\draw[->-=0.8] (G4A2) -- (q2n);
\draw[->-=1] (G4b) -- (p3n);

%GAMMAm
\tikzstyle{every node}=[inner sep=1pt]
\node at (7.75,2.15) {$\Gamma_{m-1}$};
\begin{tiny}
\node (GmA1) at ($(p_{m-2}n) + (0.55,-0.5)$) {$A$};
\node (GmA2) at ($(q_{m-2}n) + (-0.1,0.8)$) {$A$};
\node (Gmb) at ($(p_{m-1}n) + (-0.45, -0.45)$) {$b$};
\end{tiny}
\draw[->-=1] (GmA1) -- (p_{m-2}n);
\draw[->-=0.8] (GmA2) -- (q_{m-2}n);
\draw[->-=1] (Gmb) -- (p_{m-1}n);

%GAMMAm1
\tikzstyle{every node}=[inner sep=1pt]
\node at (9.5,1.8) {$\Gamma_m$};
\begin{tiny}
\node (G_{m+1}A1) at ($(p_{m-1}n) + (0.5,-0.6)$) {$A$};
\node (G_{m+1}A2) at ($(q_{m-1}n) + (0.15,0.75)$) {$A$};
\node (G_{m+1}b) at ($(b) + (-0.5, 0.0)$) {$b$};
\end{tiny}
\draw[->-=1] (G_{m+1}A1) -- (p_{m-1}n);
\draw[->-=0.8] (G_{m+1}A2) -- (q_{m-1}n);
\draw[->-=1] (G_{m+1}b) -- (bn);

\end{tikzpicture}
\caption{
Game division in Case~\ref{case:final}.
In this figure $q_{m-1} \neq q_{m}$, $l=m$, and $c=b^{+}$.
Path $P$ is depicted in red, path $Q$ is depicted in blue.
}
\label{fig:final_3}
\end{figure}
\begin{figure}[h]
\begin{tikzpicture}[scale=1, yscale=0.85]

\tikzset{->-/.style={decoration={
  markings,
  mark=at position #1 with {\arrow{latex}}},postaction={decorate}}}

\coordinate (a1) at (0,0) {};
\coordinate (a2) at (0.3,2) {};
\coordinate (p2) at (2.5,3.3) {};
\coordinate (p3) at (4.5,3.65) {};
\coordinate (up) at (5.5,3.7) {};
\coordinate (p4) at (6.5,3.65) {};
\coordinate (b) at (10.7,2) {};
\coordinate (b+) at (11,0) {};

\coordinate (q01) at (0.75,0) {};
\coordinate (q1) at (1.5,0) {};
\coordinate (q11) at (2.2,0) {};
\coordinate (q12) at (2.8,0) {};
\coordinate (q2) at (3.5,0) {};
\coordinate (q21) at (4.4,0) {};
\coordinate (q3) at (5.3,0) {};

\coordinate (q31) at (6.05,-0.4) {};

\coordinate (c') at (6.8,-1.975);
\coordinate (c) at (8.5,-1.75);
\coordinate (c'c) at (7.4,-0.8);

\coordinate (low) at (5.5,-2.0) {};

\filldraw[fill=gray!25, draw=none]
  (q1) -- (a2) to[out=50,in=200] (p2) -- cycle;

\filldraw[fill=gray!10, draw=none]
  (q2) -- (p2) to[out=15,in=187] (p3) -- cycle;

\filldraw[fill=gray!25, draw=none]
  (q3) -- (p3) to[out=3,in=177] (p4) -- cycle;

\filldraw[fill=gray!10, draw=none]
  (c) -- (p4) to[out=-4,in=130] (b) -- (b+) to[out=240, in=10] cycle;

\filldraw[fill=gray!50, draw=none]
  (a1) -- (q3) to[out=-10, in=110] (c') to[out=180, in=0] (low) to[out=180, in=-60] cycle;
\filldraw[fill=gray!50, draw=none]
  (c') to[out=85, in=190] (c'c) to[out=-2, in=110] (c) to[out=190, in=4] cycle;

\draw[thick] (a1) -- (a2);
\draw[thick] (a2) to[out=50,in=200] (p2);
\draw[thick] (p2) to[out=15,in=187] (p3);
\draw[thick] (p3) to[out=3,in=177] (p4);
\draw[thick] (p4) to[out=-4,in=130] (b);
\draw[thick] (b) -- (b+);

\draw[thick] (a1) -- (q3);
\draw[thick] (q3) to[out=-10, in=110] (c');

\draw[thick] (b+) to[out=240, in=18] (c);
\draw[thick] (c) to[out=190, in=4] (c');
\draw[thick] (c') to[out=180, in=0] (low);
\draw[thick] (low) to[out=180, in=-60] (a1);

\draw[thick] (c') to[out=85, in=190] (c'c);
\draw[thick] (c'c) to[out=-2, in=110] (c);

\draw[->-=.6, thick] (q1) to (a2);
\draw[->-=.5] (q01) to (a2);

\draw[thick] (q1) to (p2);
\draw[->-=.5] (q11) to (p2);
\draw[->-=.5] (q12) to (p2);
\draw[->-=.55, thick] (q2) to (p2);

\draw[thick] (q2) to (p3);
\draw[->-=.5] (q21) to (p3);
\draw[->-=.5,thick] (q3) to (p3);
\draw[thick] (q3) to (p4);
\draw[->-=.40] (q31) to (p4);
\draw[->-=.555] (c') to (p4);
\draw[->-=.43] (c'c) to (p4);
\draw[->-=.55,thick] (c) to (p4);

\begin{tiny}
\tikzstyle{every node}=[inner sep=0pt]
\node at ($(a1) + (-0.65,0)$) {$a_1 = q_0$};
\node at ($(a2) + (-0.65,0)$) {$a_2 = p_1$};
\node at ($(p2) + (0,0.3)$) {$p_2$};
\node at ($(p3) + (0,0.3)$) {$p_3$};
\node at ($(p4) + (0,0.3)$) {$p_4 = p_l$};
\node at ($(b) + (0.3,0)$) {$b$};
\node at ($(b+) + (0.4,0)$) {$b^+$};
\node at ($(q1) + (0,-0.35)$) {$q_1$};
\node at ($(q2) + (0,-0.35)$) {$q_2$};
\node at ($(q3) + (-0.1,-0.35)$) {$q_3$};
\node at ($(c) + (0.0,-0.25)$) {$c$};
\node at ($(c') + (0.0,-0.3)$) {$c'$};
\end{tiny}

\begin{tiny}
\tikzstyle{every node}=[circle,minimum size=5pt,inner sep=0pt,draw,fill]
\node (a1n) at (a1) {};
\node (a2n) at (a2) {};
\node (p2n) at (p2) {};
\node (p3n) at (p3) {};
\node (p4n) at (p4) {};
\node (bn) at (b) {};
\node (b+n) at (b+) {};

\node (q1n) at (q1) {};
\node (q2n) at (q2) {};
\node (q3n) at (q3) {};
\node (cn) at (c) {};
\node (c'n) at (c') {};

\end{tiny}

%GAMMA2
\tikzstyle{every node}=[inner sep=1pt]
\node at (3.7,-1.1) {$\Gamma_1$};
\begin{tiny}
\node (G2A) at ($(a1n) + (0.55,-0.3)$) {$A$};
\node (G2b) at ($(c) + (-0.5, +0.25)$) {$b$};
\end{tiny}
\draw[->-=1] (G2A) -- (a1n);
\draw[->-=1] (G2b) -- (cn);

%GAMMA3
\tikzstyle{every node}=[inner sep=1pt]
\node at (1.5,1.8) {$\Gamma_2$};
\begin{tiny}
\node (G3A1) at ($(a2n) + (0.65,-0.1)$) {$A$};
\node (G3A2) at ($(q1n) + (-0.05,0.7)$) {$A$};
\node (G3b) at ($(p2n) + (-0.45, -0.55)$) {$b$};
\end{tiny}
\draw[->-=1] (G3A1) -- (a2n);
\draw[->-=0.8] (G3A2) -- (q1n);
\draw[->-=1] (G3b) -- (p2n);

%GAMMA4
\tikzstyle{every node}=[inner sep=1pt]
\node at (3.5,2.1) {$\Gamma_3$};
\begin{tiny}
\node (G4A1) at ($(p2n) + (0.55,-0.3)$) {$A$};
\node (G4A2) at ($(q2n) + (0.0,0.8)$) {$A$};
\node (G4b) at ($(p3n) + (-0.45, -0.45)$) {$b$};
\end{tiny}
\draw[->-=1] (G4A1) -- (p2n);
\draw[->-=0.8] (G4A2) -- (q2n);
\draw[->-=1] (G4b) -- (p3n);

%GAMMA5
\tikzstyle{every node}=[inner sep=1pt]
\node at (5.5,2.2) {$\Gamma_4$};
\begin{tiny}
\node (G5A1) at ($(p3n) + (0.55,-0.4)$) {$A$};
\node (G5A2) at ($(q3n) + (0.05,0.8)$) {$A$};
\node (G5b) at ($(p4n) + (-0.5, -0.4)$) {$b$};
\end{tiny}
\draw[->-=1] (G5A1) -- (p3n);
\draw[->-=0.8] (G5A2) -- (q3n);
\draw[->-=1] (G5b) -- (p4n);

%GAMMA6
\tikzstyle{every node}=[inner sep=1pt]
\node at (9.2,1) {$\Gamma_5$};
\begin{tiny}
\node (G6A1) at ($(p4n) + (0.5,-0.4)$) {$A$};
\node (G6A2) at ($(cn) + (0.4,0.5)$) {$A$};
\node (G6b) at ($(b) + (-0.5, -0.1)$) {$b$};
\end{tiny}
\draw[->-=1] (G6A1) -- (p4n);
\draw[->-=1] (G6A2) -- (cn);
\draw[->-=1] (G6b) -- (bn);

\end{tikzpicture}
\caption{
Game division in Case~\ref{case:final}.
In this figure $c \neq b^+$.
}
\label{fig:final_1}
\end{figure}
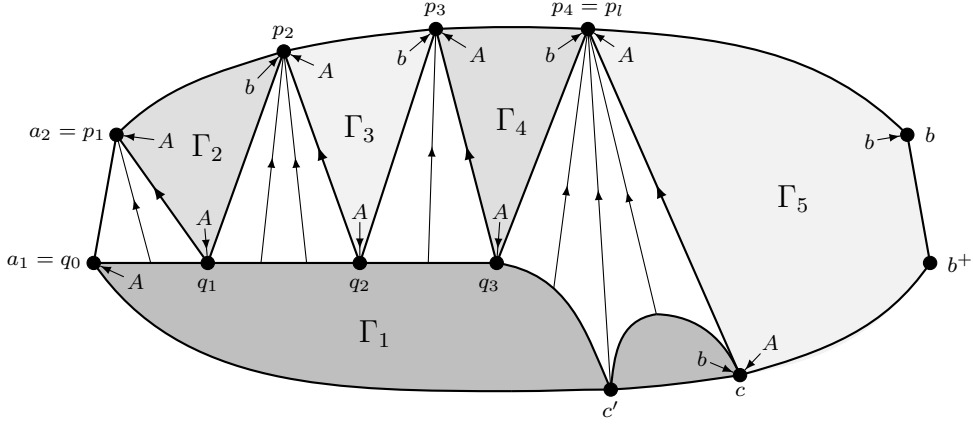

\Case{Final case}\label{case:final}
%  Figures~\ref{fig:final_1} to~\ref{fig:final_3} depict the game division that we use in this case.
  Let $P$ denote the path $C[a_2,b]$.
  Let $Q$ denote the unique longest simple path from $a_1$ to $b^+$ in $G \grminus P$ that traverses only vertices adjacent to $P$ in $G$.

  Let $p_1=a_2$, and let $p_2, p_3,\ldots, p_{m-1}$ be the set of all interior vertices of path $P$ that have at least two neighbors in $Q$, and occur in this order in $P$, and let $p_m=b$.
%  Observe that for $i < j-1$, vertex $p_i$ is not adjacent to $p_j$.
%  Indeed, an edge connecting $p_i$ and $p_j$ would have to intersect with edges connecting $p_{i+1}$ with $Q$.
  As $G$ is near-triangulated, for $i=1,\ldots, m-1$, vertices $p_i$ and $p_{i+1}$ have a unique common neighbor in $Q$.
  Let $q_0=a_1$, $q_m=b^+$, and for $i=1,\ldots, m-1$, let $q_i$ be the common neighbor of $p_i$ and $p_{i+1}$ in $Q$.
  Note that $q_1,\ldots,q_{m-1}$ are pairwise different.
  Moreover, if $q_0 = q_1$ then $\set{a_1,p_2}$ is a chord of $C$ that separates $a_2$ from $b$ and we can apply Case~\ref{case:chord}.
  However, it is possible that $q_{m-1}=q_m$.
  See Figure~\ref{fig:final_2} to see $q_{m-1}=q_m$, and Figure~\ref{fig:final_3} to see $q_{m-1} \neq q_m$.
  In Figures~\ref{fig:final_2},~\ref{fig:final_3} path $Q$ does not intersect $C(b^+,a_1)$.
  See Figure~\ref{fig:final_1} to see a non-empty intersection of $Q$ and $C(b^+,a_1)$.

  Observe, that for $i=1,\ldots,m-1$, vertex $q_i$ is not adjacent to any vertex $p_{j}$ other than $p_{i}$ and $p_{i+1}$.
  Indeed, an edge connecting $q_i$ with $p_j$ for $j < i$ would have to intersect with edges connecting $p_{i}$ with $Q$.
  Similarly, an edge connecting $q_i$ with $p_j$ for $j > i+1$ would have to intersect with edges connecting $p_{i+1}$ with $Q$.
  Each vertex in $Q(q_{i-1},q_i)$ is not adjacent to any vertex in $P$ other than $p_{i}$.
  Each vertex in $P(p_{i},p_{i+1})$ is not adjacent to any vertex in $Q$ other than $q_{i}$.
  Moreover, the definition of $Q$ guarantees that there are no vertices in $\inter{Q[q_i,q_{i+1}],p_{i+1}}$ other than $Q[q_i,q_{i+1}]$, and $p_{i+1}$.
  %Since $G$ is near-triangulated, each $p_i$ is adjacent to all vertices in $Q[q_{i-1},q_{i}]$.
  %On the other hand, $q_i$ might be not adjacent to some vertex of $P[p_{i-1},p_i]$, and then $q_i$ is a cut-vertex of $G \grminus P$.

  Let $p$ be the first (closest to $a_2$) vertex on path $P$ such that $p$ is adjacent to a vertex in $C[b^+,a_1)$.
  The vertex $p$ exists since $p_m=b$ is adjacent to $b^+$.
  Let $c$ be the first (closest to $b^+$) neighbor of $p$ on the path $C[b^+,a_1)$ and observe that $c$ is a vertex of $Q$.
  Let $l$ be the minimal $l$ such that $c = q_{l}$ or that $c$ is in $Q(q_l,q_{l+1})$ and observe that $p_l$ is the first (closest to $a_2$) vertex on path $P$ that is adjacent to $c$.
  Thus, we have $p=p_l$.

  If $p_l = p_1$ then $\set{a_2,c}$ is a chord of $C$ that separates $a_1$ from $b$ and we can apply Case~\ref{case:chord}.
  Thus, we can assume that $2 \leq l \leq m$.
  In case $l=m$ we have that $c=b^+$.
  See Figure~\ref{fig:final_3}.
  For $l=m-1$ and $q_{m-1}=q_m=b^+$, we also have $c=b^+$.
  See Figure~\ref{fig:final_2}.
  Otherwise, we have $c \neq b^+$.
  See Figure~\ref{fig:final_1}.
  Observe, that in any case we have that $Q(a_1,q_{l-1}]$ does not intersect $C$.
  On the other hand, $Q(q_{l+1},c)$ might intersect $C$ if $p_l$ has more than one neighbor in $C$.

  Let $G_1$ be $\inter{C[c,a_1], Q(a_1,c]}$.
  Let $G_i$, for $i=2,4,\ldots,l$, be $\inter{p_{i-1},p_{i},q_{i-1},p_{i-1}}$.
  Additionally, if $l<m$, let $G_{l+1}$ be $\inter{C[p_l,c],p_l}$.
  If $l=m$, then $G_{l+1}$ is not defined.
  Observe that any neighbor of $p_l$ in $C$ other than $c$ is an $(\set{a_1},c)$-cut in $G_1$.

  For each vertex $p_i$, for $i=2,\ldots,\min(l,m-1)$, we devalue the token function by removing the token of value $2$.
  This way we obtain that all vertices $p_i$ for $i=1,\ldots,l$ have one token.
  We divide the game into smaller games:
  \begin{itemize}
    \item $\Gamma_1=(G_1,\set{a_1},c)$,
    \item $\Gamma_i=(G_i,\set{p_{i-1},q_{i-1}},p_{i})$, for $i=2,4\ldots,l$,
    \item $\Gamma_{l+1}=(G_{l+1}, \set{p_l,c}, b)$ (if $l=m$, $\Gamma_{l+1}$ is not defined).
  \end{itemize}
  For $i=1,\ldots,l-1$, each vertex in $Q(q_{i-1},q_i]$ gives away a token to $p_i$.
  Each vertex in $Q(q_{l-1},c]$ gives away a token to $p_l$.
  For $i=1,\ldots,l-1$, vertex $q_i$ different than $c$ has one token of value $2$ and one token of value $3$ in the game $\Gamma_1$.
  We mark such a vertex with a token of value $2$ when $p_{i+1}$ is colored in the same round.
  Otherwise it is marked with a token of value $3$ in $\Gamma_1$.

  Vertex $a_1$ gets at most one defect in game $\Gamma_1$.
  Vertex $a_2$ gets at most one defect in game $\Gamma_2$.
  Vertex $b$ gets at most one defect in game $\Gamma_{l+1}$ when $l<m$, and at most one defect in game $\Gamma_{m}$ when $l=m$.
  For $i=2,\ldots,m-1$, vertex $p_i$ gets at most one defect in $\Gamma_{i}$, and at most one defect in $\Gamma_{i+1}$.
  The strategy that chooses the value of a token removed from vertex $q_i$ in $\Gamma_1$ guarantees that each vertex $q_i$ receives at most three defects in total.

\end{proof}

\section{Lister's strategy}\label{sec:lister}

In this section we show a planar graph which is not $2$-defective $3$-paintable.
We begin with a definition of a family of outerplanar graphs that play a crucial role in the construction.

An \emph{$l$-layered, $k$-petal daisy $D(l,k)$} is an outerplanar graph with the vertex set partitioned into $l$ layers,
$L_1,\ldots,L_{l}$, defined inductively as follows:
\begin{itemize}
  \item $1$-layered, $k$-petal daisy $D(1,k)$ is a single edge $\set{u,v}$, and $L_1 = \set{u,v}$.
  \item $l$-layered, $k$-petal daisy $D(l,k)$ for $l>1$ extends $D(l-1,k)$ in the following way:
  for every edge $\set{u,v}$ of $D(l-1,k)$ with $u,v \in L_{l-1}$
  we add a path $P(u,v)$ on $2k-1$ new vertices and join the first $k$ vertices of $P(u,v)$ to $u$ and join the last $k$ vertices of $P(u,v)$ to $v$.
  The \emph{inner} vertices of path $P(u,v)$ are all the vertices of $P(u,v)$ except the two end-points.
  We set $L_{l} = \bigcup \set{P(u,v): \set{u,v} \text{ is an edge with both endpoints in } L_{l-1}}$.
\end{itemize}
We draw $D(l,k)$ in an outerplanar way, i.e., such that all vertices are adjacent to the outerface. In particular, in such a drawing
all inner faces of $D(l,k)$ are triangles -- see Figure~\ref{fig:daisy} for an example.

\begin{figure}[h]
\centering
\begin{tikzpicture}[scale=0.5,yscale=0.7]

  \tikzstyle{every node}=[inner sep=2pt,fill=white]
  \draw[dashed] (-2,11)--(26,11);
  \draw[dashed] (-2,9)--(26,9);
  \node at (23.5,10) {$L_1$};

  \draw[dashed] (-2,6)--(26,6);
  \draw[dashed] (-2,4)--(26,4);
  \node at (23.5,5) {$L_2$};

  \draw[dashed] (-2,1)--(26,1);
  \draw[dashed] (-2,-1)--(26,-1);
  \node at (23.5,0) {$L_3$};

\begin{tiny}
\node at (-1.1,0) {$P(u,v)$};
\node at (0.6,4.6) {$u$};
\node at (4.4,4.6) {$v$};
\end{tiny}

  \tikzstyle{every node}=[circle,minimum size=5pt,inner sep=0pt,draw,fill]

  %wierzcholki poziomu 1
  \node (l1-p0) at (6,10) {};
  \node (l1-p1) at (14,10) {};

  %wierzcholki poziomu 1
  \node (l2-p0p1-p0) at (0,5) {};
  \node (l2-p0p1-p1) at (5,5) {};
  \node (l2-p0p1-p2) at (10,5) {};
  \node (l2-p0p1-p3) at (15,5) {};
  \node (l2-p0p1-p4) at (20,5) {};

  %wierzcholki poziomu 2
  \node (l3-p0p1-p0p1-p0) at (0.5,0) {};
  \node (l3-p0p1-p0p1-p1) at (1.5,0) {};
  \node (l3-p0p1-p0p1-p2) at (2.5,0) {};
  \node (l3-p0p1-p0p1-p3) at (3.5,0) {};
  \node (l3-p0p1-p0p1-p4) at (4.5,0) {};

  \node (l3-p0p1-p1p2-p0) at (5.5,0) {};
  \node (l3-p0p1-p1p2-p1) at (6.5,0) {};
  \node (l3-p0p1-p1p2-p2) at (7.5,0) {};
  \node (l3-p0p1-p1p2-p3) at (8.5,0) {};
  \node (l3-p0p1-p1p2-p4) at (9.5,0) {};

  \node (l3-p0p1-p2p3-p0) at (10.5,0) {};
  \node (l3-p0p1-p2p3-p1) at (11.5,0) {};
  \node (l3-p0p1-p2p3-p2) at (12.5,0) {};
  \node (l3-p0p1-p2p3-p3) at (13.5,0) {};
  \node (l3-p0p1-p2p3-p4) at (14.5,0) {};

  \node (l3-p0p1-p3p4-p0) at (15.5,0) {};
  \node (l3-p0p1-p3p4-p1) at (16.5,0) {};
  \node (l3-p0p1-p3p4-p2) at (17.5,0) {};
  \node (l3-p0p1-p3p4-p3) at (18.5,0) {};
  \node (l3-p0p1-p3p4-p4) at (19.5,0) {};

  %sciezki poziomu 1
  \path (l1-p0) edge[thick] (l1-p1);

  %sciezki poziomu 2
  \path (l2-p0p1-p0) edge[thick] (l2-p0p1-p1);
  \path (l2-p0p1-p1) edge[thick] (l2-p0p1-p2);
  \path (l2-p0p1-p2) edge[thick] (l2-p0p1-p3);
  \path (l2-p0p1-p3) edge[thick] (l2-p0p1-p4);

  %sciezki poziomu 3
  \path (l3-p0p1-p0p1-p0) edge[thick] (l3-p0p1-p0p1-p1);
  \path (l3-p0p1-p0p1-p1) edge[thick] (l3-p0p1-p0p1-p2);
  \path (l3-p0p1-p0p1-p2) edge[thick] (l3-p0p1-p0p1-p3);
  \path (l3-p0p1-p0p1-p3) edge[thick] (l3-p0p1-p0p1-p4);

  \path (l3-p0p1-p1p2-p0) edge[thick] (l3-p0p1-p1p2-p1);
  \path (l3-p0p1-p1p2-p1) edge[thick] (l3-p0p1-p1p2-p2);
  \path (l3-p0p1-p1p2-p2) edge[thick] (l3-p0p1-p1p2-p3);
  \path (l3-p0p1-p1p2-p3) edge[thick] (l3-p0p1-p1p2-p4);

  \path (l3-p0p1-p2p3-p0) edge[thick] (l3-p0p1-p2p3-p1);
  \path (l3-p0p1-p2p3-p1) edge[thick] (l3-p0p1-p2p3-p2);
  \path (l3-p0p1-p2p3-p2) edge[thick] (l3-p0p1-p2p3-p3);
  \path (l3-p0p1-p2p3-p3) edge[thick] (l3-p0p1-p2p3-p4);

  \path (l3-p0p1-p3p4-p0) edge[thick] (l3-p0p1-p3p4-p1);
  \path (l3-p0p1-p3p4-p1) edge[thick] (l3-p0p1-p3p4-p2);
  \path (l3-p0p1-p3p4-p2) edge[thick] (l3-p0p1-p3p4-p3);
  \path (l3-p0p1-p3p4-p3) edge[thick] (l3-p0p1-p3p4-p4);

  %krawedzie z poziomu 1 do poziomu 2
  \path (l1-p0) edge[thick] (l2-p0p1-p0);
  \path (l1-p0) edge[thick] (l2-p0p1-p1);
  \path (l1-p0) edge[thick] (l2-p0p1-p2);
  \path (l1-p1) edge[thick] (l2-p0p1-p2);
  \path (l1-p1) edge[thick] (l2-p0p1-p3);
  \path (l1-p1) edge[thick] (l2-p0p1-p4);

  %krawedzie z poziomu 2 do poziomu 3
  \path (l2-p0p1-p0) edge[thick] (l3-p0p1-p0p1-p0);
  \path (l2-p0p1-p0) edge[thick] (l3-p0p1-p0p1-p1);
  \path (l2-p0p1-p0) edge[thick] (l3-p0p1-p0p1-p2);
  \path (l2-p0p1-p1) edge[thick] (l3-p0p1-p0p1-p2);
  \path (l2-p0p1-p1) edge[thick] (l3-p0p1-p0p1-p3);
  \path (l2-p0p1-p1) edge[thick] (l3-p0p1-p0p1-p4);

  \path (l2-p0p1-p1) edge[thick] (l3-p0p1-p1p2-p0);
  \path (l2-p0p1-p1) edge[thick] (l3-p0p1-p1p2-p1);
  \path (l2-p0p1-p1) edge[thick] (l3-p0p1-p1p2-p2);
  \path (l2-p0p1-p2) edge[thick] (l3-p0p1-p1p2-p2);
  \path (l2-p0p1-p2) edge[thick] (l3-p0p1-p1p2-p3);
  \path (l2-p0p1-p2) edge[thick] (l3-p0p1-p1p2-p4);

  \path (l2-p0p1-p2) edge[thick] (l3-p0p1-p2p3-p0);
  \path (l2-p0p1-p2) edge[thick] (l3-p0p1-p2p3-p1);
  \path (l2-p0p1-p2) edge[thick] (l3-p0p1-p2p3-p2);
  \path (l2-p0p1-p3) edge[thick] (l3-p0p1-p2p3-p2);
  \path (l2-p0p1-p3) edge[thick] (l3-p0p1-p2p3-p3);
  \path (l2-p0p1-p3) edge[thick] (l3-p0p1-p2p3-p4);

  \path (l2-p0p1-p3) edge[thick] (l3-p0p1-p3p4-p0);
  \path (l2-p0p1-p3) edge[thick] (l3-p0p1-p3p4-p1);
  \path (l2-p0p1-p3) edge[thick] (l3-p0p1-p3p4-p2);
  \path (l2-p0p1-p4) edge[thick] (l3-p0p1-p3p4-p2);
  \path (l2-p0p1-p4) edge[thick] (l3-p0p1-p3p4-p3);
  \path (l2-p0p1-p4) edge[thick] (l3-p0p1-p3p4-p4);

  \draw[rounded corners=0.2cm] (0,0)--(0,0.5)--(5,0.5)--(5,-0.5)--(0,-0.5)--(0,0);

\end{tikzpicture}
\caption{A 3-layered 3-petal daisy $D(3,3)$.}
\label{fig:daisy}
\end{figure}
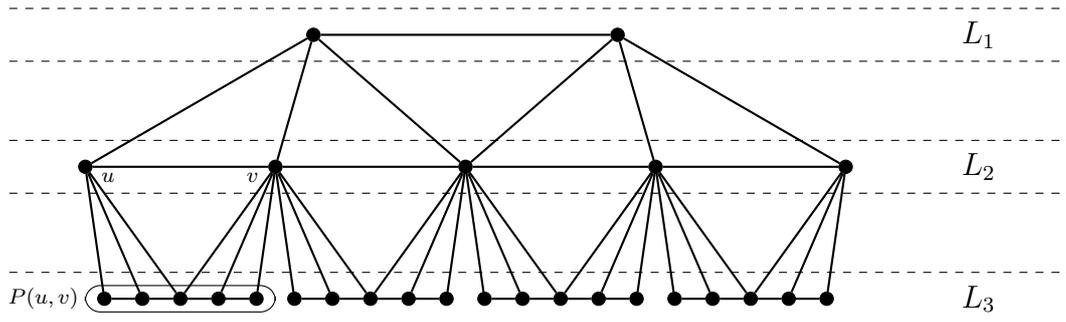

A planar graph $G$ is an \emph{edge extension of $D(l,k)$} if $G$ extends $D(l,k)$ in the following way: for every inner face $F$ of $D(l,k)$ we add a vertex $v(F)$ that is adjacent to some two vertices on the boundary of $F$.
A planar graph $G$ is the \emph{face extension of $D(l,k)$} if $G$ extends $D(l,k)$ in the following way: for every inner face $F$ of $D(l,k)$ we add a set $u(F)$ of four vertices such that one vertex in $u(F)$, say $u$, is adjacent to all vertices on the boundary of $F$, and for every edge $e$ on the boundary of $F$, one vertex in $u(F)$ is adjacent to $u$ and to the endpoints of $e$.
If $G$ is an edge/face extension of $D(l,k)$, the copy of $D(l,k)$ in $G$ is called \emph{the skeleton} of $G$ and is denoted $\skel(G)$.
Let $\ext(G)$ denote the vertices in $G$ that are not in $\skel(G)$.
See Figure~\ref{fig:extensions} for examples of an edge extension and of a face extension.

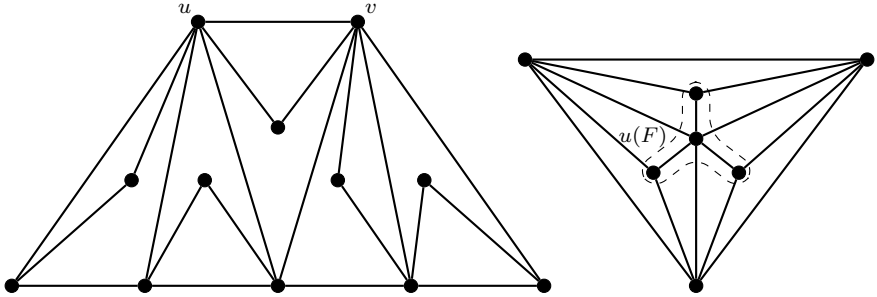
\begin{figure}[h]

\begin{tikzpicture}
\begin{scope}[xscale=0.35,yscale=0.35]

\tikzstyle{every node}=[inner sep=2pt,fill=white]
  \begin{tiny}
    \node at (6.5, 10.5) {$u$};
    \node at (13.5, 10.5) {$v$};
%    \node at (10, 9.25) {$F$};
%    \node at (10, 5) {$v(F)$};
%    \node at (2.5, 0.75) {$F_1$};
%    \node at (7.5, 0.75) {$F_2$};
%    \node at (12.5, 0.75) {$F_3$};
%    \node at (15.5, 0.75) {$F_4$};
  \end{tiny}

  \tikzstyle{every node}=[circle,minimum size=5pt,inner sep=0pt,draw,fill]

  %wierzcholki poziomu 1

  \node (u) at (7,10) {};
  \node (v) at (13,10) {};

  \node (p1) at (0,0) {};
  \node (p2) at (5,0) {};
  \node (p3) at (10,0) {};
  \node (p4) at (15,0) {};
  \node (p5) at (20,0) {};

  \node (f1) at (4.5,4) {};
  \node (f2) at (7.25,4) {};
  \node (f3) at (10,6) {};
  \node (f4) at (12.25,4) {};
  \node (f5) at (15.5,4) {};

  \draw[thick] (u) edge[thick] (v);

  \path (p1) edge[thick] (p2);
  \path (p2) edge[thick] (p3);
  \path (p3) edge[thick] (p4);
  \path (p4) edge[thick] (p5);

  \path (u) edge[thick] (p1);
  \path (u) edge[thick] (p2);
  \path (u) edge[thick] (p3);
  \path (v) edge[thick] (p3);
  \path (v) edge[thick] (p4);
  \path (v) edge[thick] (p5);

  \path (p1) edge[thick] (f1);
  \path (u) edge[thick] (f1);

  \path (p2) edge[thick] (f2);
  \path (p3) edge[thick] (f2);

  \path (u) edge[thick] (f3);
  \path (v) edge[thick] (f3);

  \path (v) edge[thick] (f4);
  \path (p4) edge[thick] (f4);

  \path (p4) edge[thick] (f5);
  \path (p5) edge[thick] (f5);
\end{scope}

\begin{scope}[xscale=0.45,yscale=0.3, shift={(15,0)}]
  \tikzstyle{every node}=[inner sep=2pt,fill=white]
  \begin{tiny}
    \node at (3.45,6.6) {$u(F)$};
%    \node at (5,9.5) {$F$};

  \end{tiny}

  \tikzstyle{every node}=[circle,minimum size=5pt,inner sep=0pt,draw,fill]
  \node (f1) at (0,10) {};
  \node (f2) at (10,10) {};
  \node (f3) at (5,0) {};

  \node (u123) at (5,6.5) {};
  \node (u12) at (5,8.5) {};
  \node (u23) at (6.25,5) {};
  \node (u13) at (3.75,5) {};

%  \node (u1) at (5,9) {};
%  \node (u2) at (5.4,9) {};
%
%  \node (m1) at (5.4,6.8) {};

% \node (r1) at (6.75,5.2) {};
%  \node (r2) at (6.25,4.25) {};

%  \node (m2) at (5,5.7) {};

%  \node (l2) at (3.75,4.25) {};

  \draw[dashed, rounded corners=0.2cm] (5,9) -- (5.4,9) --(5.4,6.75)--(6.75,5.2) -- (6.25,4.25) -- (5,5.7) -- (3.75,4.25) -- (3.25,5.2) -- (4.6,6.75) -- (4.6,9)--(5,9);

  \tikzstyle{every node}=[inner sep=2pt,fill=white]
%  \begin{tiny}
%     \node at (5,9.25) {$u_{1,2}$};
%     \node at (5.25,7.15) {$u$};
%     \node at (3.75,5.95) {$u_{1,3}$};
%     \node at (6.25,5.95) {$u_{2,3}$};
%  \end{tiny}

  \path (f1) edge[thick] (f2);
  \path (f2) edge[thick] (f3);
  \path (f3) edge[thick] (f1);

  \path (f1) edge[thick] (u123);
  \path (f2) edge[thick] (u123);
  \path (f3) edge[thick] (u123);

  \path (f1) edge[thick] (u12);
  \path (f2) edge[thick] (u12);

  \path (f2) edge[thick] (u23);
  \path (f3) edge[thick] (u23);

  \path (f1) edge[thick] (u13);
  \path (f3) edge[thick] (u13);

  \path (u123) edge[thick] (u12);
  \path (u123) edge[thick] (u13);
  \path (u123) edge[thick] (u23);

\end{scope}
\end{tikzpicture}

\caption{On the left: an edge extension of $D(2,3)$. On the right: the face extension of $D(2,1)$ (which has only one inner face $F$).}
\label{fig:extensions}
\end{figure}

\begin{lemma}\label{lem:outerplanar-not-2-defective-2-choosable}
  Any edge extension of $D(l,k)$ for $l = 4$ and $k = 362$ is not $2$-defective $2$-choosable.
\end{lemma}

Before the proof of Lemma~\ref{lem:outerplanar-not-2-defective-2-choosable} we show how to use it to construct a planar graph that is not $2$-defective $3$-paintable.

\begin{proof}[Proof of Theorem~\ref{thm:negative}]
Fix $l=4$ and $k=362$.
Let $G_1,\ldots,G_9$ be nine copies of the face extension of $D(l,k)$.
%Let $G'_1,\ldots,G'_9$ be the face extensions of $G_1,\ldots,G_9$.
Let $G$ be a planar graph that is formed of $G_1,\ldots,G_9$ and a vertex $v$ joined to every vertex in $\skel(G_1),\ldots,\skel(G_9)$.
We show a winning strategy for Lister in a $2$-defective $3$-painting game on $G$.

In the first four rounds Lister plays the following strategy:
\begin{itemize}
  \item in the $i$-th round, for $i=1,2,3$, Lister marks $v$, if it is still uncolored, and the vertices in $\skel(G_{3i-2})$, $\skel(G_{3i-1})$, and $\skel(G_{3i})$.
  \item in the $4$-th round, Lister marks all the vertices in $\ext(G_1),\ldots,\ext(G_9)$.
\end{itemize}

Clearly, Painter needs to color vertex $v$ in one of the first three rounds.
Say, he colors $v$ in the $i$-th round.
All vertices from the skeletons of $G_{3i-2}$, $G_{3i-1}$, and $G_{3i}$ are adjacent to $v$ and at most two of them are colored in the $i$-th round.
Let $H$ be a graph, one among $G_{3i-2}$, $G_{3i-1}$, $G_{3i}$, such that no vertex of $H$ is colored in the $i$-th round.
Observe that after three rounds, all vertices from $skel(H)$ are uncolored and have only two tokens left.

In the fourth round, for any inner face $F$ of $skel(H)$, Painter colors at most three vertices in $u(F)$.
Thus, for every inner face $F$ of $skel(H)$, at least one vertex in $u(F)$ is still uncolored and has only two tokens left.

Let $H'$ be the graph induced by the set of uncolored vertices in $H$ after the $4$-th round.
Clearly, $H'$ is a supergraph of some edge extension of $D(l,k)$ and each vertex in $H'$ has two tokens left.
The state of the game on $H'$ is the same as the initial state of a 2-defective 2-painting game on $H'$.

By Lemma~\ref{lem:outerplanar-not-2-defective-2-choosable}, graph $H'$ is not $2$-defective $2$-choosable and hence Lister has a winning strategy in the 2-defective 2-painting game on $H'$.
\end{proof}

\begin{proof}[Proof of Lemma~\ref{lem:outerplanar-not-2-defective-2-choosable}]
Fix $l=4$ and $k=362$.
  Let $G$ be an edge extension of $D(l,k)$. For notational convenience, let $D(l,k)$ denote the skeleton of $G$, and use the notation introduced in the definition of $D(l,k)$.

We split all vertices of the first three layers of $D(l,k)$ into two categories.
A vertex $x \in L_i$ for $i < l$ is \emph{bad} if there exist: a vertex $y \in L_i$ adjacent to $x$; an inner vertex $z$ of the path $P(x,y)$; and a vertex in $\ext(G)$ that is adjacent both to $x$ and $z$.
Otherwise, $x$ is \emph{good}.
For example, in Figure~\ref{fig:extensions}, vertex $u$ is good, while vertex $v$ is bad.

Let $z$ be a vertex in $L_i$, $i \in \sbrac{2}$.
Note that the neighborhood of $z$ in $L_{i+1}$ induces one or two paths of size $k$ in $G$: we denote them $P_1(z)$, and $P_2(z)$.
If neighborhood of $z$ in $L_{i+1}$ induces only one path then $P_2(z)$ is undefined.

We claim that if some vertex $z$ in $L_1, L_2$ has at least $15$ bad neighbors in $P_j(z)$ for some $j \in \sbrac{2}$ then $G$ is not $2$-defective $2$-colorable.
Suppose to the contrary that $z$ has $15$ bad neighbors in $P_j(z)$ for some $j \in \sbrac{2}$ and that there is a $2$-defective coloring of $G$ with colors $\alpha$ and $\beta$.
Without loss of generality, $z$ is colored $\alpha$.
Among the neighbors of $z$ in $P_j(z)$ at most two are colored $\alpha$.
Vertices colored $\alpha$ in $P_j(z)$ split $P_j(z)$ into at most three subpaths that consist only of vertices colored $\beta$.
As there are at least $15$ bad vertices in $P_j(z)$ there is a subpath $P$ of $P_j(z)$ that consists of $5$ vertices colored $\beta$ such that the middle vertex of $P$ is bad.
Let $x$ be the the middle vertex of $P$ and $y, y'$ be the two neighbors of $x$ in $P$.
As each of the vertices $y,x,y'$ has two neighbors colored $\beta$ in $P$, all vertices in $P(x,y)$, and all vertices in $P(x,y')$ are colored $\alpha$.
Since $x$ is bad, there is a vertex $w$ in $\ext(G)$ that is adjacent to $x$ and to some inner vertex $t$ in $P(x,y)$ or in $P(x,y')$.
If $w$ is colored $\alpha$ then $t$ has three neighbors colored $\alpha$.
If $w$ is colored $\beta$ then $x$ has three neighbors colored $\beta$.
So, the considered coloring is not $2$-defective, a contradiction.

For the rest of the proof we assume that every vertex $x$ in layers $L_1,L_2$ of $D(l,k)$ has at most $14$ bad neighbors in $P_j(x)$, $j \in \sbrac{2}$.
Let $x$ be a good vertex in $L_2$ (such a vertex $x$ exists as $k > 14$) and let $y$ be any neighbor of $x$ in $L_2$.
Let $W$ be a path of $24$ good neighbors of $x$ that are inner vertices of $P(x,y)$.
Such a path $W$ exists as bad neighbors of $x$ split $P(x,y)$ into at most $15$ subpaths of good vertices.
For $k=362$, one of those subpaths has at least $24$ vertices.
We number the consecutive elements of $W$ by $w_1,\ldots,w_{24}$ according to the order they appear on the path $P(x,y)$.
Now, for $i \in \sbrac{23}$, we denote the following vertices:
\begin{itemize}
  \item $c_{i}$ -- a common neighbor of $w_{i}$ and $w_{i+1}$ in the path $P(w_i,w_{i+1})$,
  \item $a_{i}$ -- a vertex $v(F)$ of the face $F$ with boundary $x,w_{i},w_{i+1}$,
  \item $b_{i}$ -- a vertex $v(F)$ of the face $F$ with boundary $c_{i},w_i,w_{i+1}$.
\end{itemize}
Note that $a_{i}$ is adjacent to $w_i$ and $w_{i+1}$ as $x$ is good, and $b_{i}$ is also adjacent to $w_i$ and $w_{i+1}$ as $w_i$ is good.
We claim that the graph $G'$ induced by the vertex set
  $$\set{x} \cup W \cup \set{a_{i}, b_{i}, c_{i} : i \in \sbrac{23}}$$
is not $2$-defective $2$-choosable, which completes the proof of the lemma.

Consider the following $2$-list assignment $L$ of $G'$ with colors $\set{\alpha, \beta, 1, \ldots, 24}$:
\begin{itemize}
  \item $L(x) = \set{\alpha, \beta}$,
  \item $L(w_i) = \set{i, \alpha}$ for $i \in \set{1,\ldots,12}$,
  \item $L(w_i) = \set{i, \beta}$ for $i \in \set{13,\ldots,24}$,
  \item $L(a_{i}) = L(b_{i}) = L(c_{i}) = \set{i,i+1}$.
\end{itemize}

Now, suppose that $c$ is a $2$-defective $L$-coloring of $G'$.
Without loss of generality we assume that $c(x) = \alpha$.
It follows that among the vertices $w_1,\ldots,w_{12}$ at most $2$ are colored $\alpha$.
Thus, there are four consecutive vertices in $w_1,\ldots,w_{12}$, say $w_j,w_{j+1},w_{j+2},w_{j+3}$ for some $j \in \sbrac{9}$, that are not colored $\alpha$.
We have $c(w_l)=l$ for $l \in \set{j,\ldots,j+3}$.
Since $c(w_{j}) = j$, at most two vertices in the set $\set{ a_{j}, b_{j}, c_{j}}$ are colored $j$.
Thus, at least one vertex in this set is colored $j+1$.
Since $c(w_{j+1}) = j+1$, at most one vertex in $\set{ a_{j+1}, b_{j+1}, c_{j+1}}$ is colored $j+1$.
Thus, at least two vertices in this set are colored $j+2$.
Eventually, since $c(w_{j+2})=j+2$, all vertices in the set $\set{ a_{j+2}, b_{j+2}, c_{j+2}}$ are colored $j+3$.
However, $w_{j+3}$ is also colored $j+3$ and $c$ is not a $2$-defective $L$-coloring.
\end{proof}

\bibliographystyle{plain}
\bibliography{paper}

\end{document}